\documentclass{amsart}
\usepackage{amsfonts,amsmath,amsthm,amssymb}
\usepackage[margin=1.4in]{geometry}
\usepackage{wasysym}
\usepackage[x11names]{xcolor}
\usepackage[shortlabels]{enumitem}
\usepackage{tikz}
\usetikzlibrary{calc}
\usepackage{dsfont}
\usepackage{caption}
\usepackage[bookmarks=false]{hyperref}

\newcommand{\R}{\mathbb{R}}
\newcommand{\Z}{\mathbb{Z}}

\newcommand{\bP}{\mathbb{P}}
\newcommand{\bE}{\mathbb{E}}

\DeclareMathOperator{\capac}{Cap}

\newtheorem{theorem}{Theorem}[section]
\newtheorem{lemma}{Lemma}[section]
\newtheorem{corollary}{Corollary}[section]
\newtheorem{proposition}{Proposition}[section]

\theoremstyle{definition}
\newtheorem{example}{Example}[section]
\newtheorem{definition}{Definition}[section]

\theoremstyle{remark}
\newtheorem{remark}{Remark}[section]

\title[$S$-transient subgraphs]{Hitting probabilities and uniformly $S$-transient subgraphs}
\author{Emily Dautenhahn}
\thanks{Partially supported by NSF grants DMS-1645643 and DMS-2054593.}
\address{Department of Mathematics, Cornell University}
\author{Laurent Saloff-Coste}
\thanks{Partially supported by NSF grant DMS-2054593.}
\address{Department of Mathematics, Cornell University}

\begin{document}

\begin{abstract}
We study the probability that a random walk started inside a subgraph of a larger graph exits that subgraph (or, equivalently, hits the exterior boundary of the subgraph). Considering the chance a random walk started in the subgraph never leaves the subgraph leads to a notion we call ``survival" transience, or $S$-transience. In the case where the heat kernel of the larger graph satisfies two-sided Gaussian estimates, we prove an upper bound on the probability of hitting the boundary of the subgraph. Under the additional hypothesis that the subgraph is inner uniform, we prove a two-sided estimate for this probability. The estimate depends upon a harmonic function in the subgraph. We also provide two-sided estimates for related probabilities, such as the harmonic measure (the chance the walk exits the subgraph at a particular point on its boundary). 
\end{abstract}

\maketitle

\section{Introduction}

In the study of  Markov chains, questions about hitting times (or exit times) of certain subsets are natural. In this paper, we are interested in discrete time random walks on countable graphs such as the square grid $\mathbb Z^d$.  Namely, we are motivated by the problem of studying random walks on  graphs that are obtained by gluing simpler graphs along particular subsets of vertices (as an example, think of $\mathbb Z^4$ and $\mathbb Z^5$ glued along their respective first coordinate axes). With this in mind, we investigate hitting times, hitting probabilities, and a related notion of transience for subgraphs of a larger graph (think  $\mathbb Z^4\setminus Z$  where $Z=\mathbb Z$ is embedded nicely into $\mathbb Z^4$) when the random walk on the underlying larger graph is assumed to have an iterated transition kernel satisfying (discrete) two-sided Gaussian estimates. We will call a graph satisfying such two-sided Gaussian estimates a Harnack graph (see Theorem \ref{VD_PI}).  There is much literature discussing two-sided Gaussian estimates on graphs and equivalent properties. See e.g., \cite{Barlow_graphs,TC_AG_GraphVol,TC_AG_Zucca_MaxPrin,Delmotte_PHI} and the references therein.

The examples we consider here stem from our goal to apply these results to the problem of gluing graphs along infinite subsets. In such settings, we are interested in whether it is certain the random walk will hit a subset $K$ or not.  We consider this as a sort of recurrence/transience question, although it is important to be careful with what is meant by these definitions. Here we define a notion of ``$S$-transience'' based on the probability $\psi_K(x)$ that a random walk started at vertex $x$ hits $K$ in finite time. This probability is of course $1 - \text{Esc}_K(x),$ where $\text{Esc}_K(x)$ is the probability $K$ is never hit. Considering the quantity $\text{Esc}_K(x)$ is related to the harmonic measure from infinity, $H_K(x)$, particularly in the case where $K$ is finite. Such questions are addressed for $\Z^d$ in Chapter 2 of \cite{lawler_intersections} and Section 25 of \cite{Spitzer}. Work of Boivin and Rau \cite{Boivin} considers the harmonic measure from infinity on weighted graphs; see also the references therein.  Moreover, questions of recurrence/transience are related to Wiener's test. However, none of these related results cover the precise situation of interest to us. 

One of our main theorems, Theorem \ref{psi-upperbd}, gives an upper bound on the hitting probability of a subset of a Harnack graph in terms of a ratio of volumes. Although this bound is not always optimal, it makes no assumptions about the geometry of the set we want to hit. This bound can be computed using only volume functions (which in practice are often easier to compute than other quantities). 

Our other main theorem, Theorem \ref{IU-psibd}, gives two-sided bounds on the hitting probability in terms of volumes and the harmonic profile $h$ (a special harmonic function). This theorem requires an additional significant geometric assumption (inner uniformity). We then obtain a partial analog to a well-known theorem that states that a Harnack space is transient (in the classical sense) if and only if \[\int^\infty [V(x,\sqrt{r})]^{-1} < + \infty\] for some/all points $x.$  Further, we can apply the same ideas as in the proof of Theorem \ref{IU-psibd} to get two-sided bounds on related quantities, such as the harmonic measure. 

The paper proceeds as follows. The rest of this section describes the setting of interest and introduces notation. Section \ref{upper} carefully defines what we mean by transience and gives an upper bound for the hitting probability of a set (Theorem \ref{psi-upperbd}). It concludes with several examples of applying the theorem to demonstrate its practicality. Section \ref{h-trans} introduces the well-known notion of $h$-transform which is used to give an upper and lower bound on the hitting probability in Theorem \ref{IU-psibd}. We also state several related corollaries and apply the theorem and corollaries to examples. Section \ref{wiener_proof} gives remarks on the relation between our results and Wiener's test. 

\subsection{General graph notation and random walks}
Let $\Gamma = (V,E)$ be a connected graph, where $E$ is a subset of the pairs of elements in $V.$ In other words, $\Gamma$ is a simple graph that does not contain loops or multiple edges. Any graphs appearing will be assumed to be simple and connected unless stated otherwise. 

On $\Gamma$ we take the usual graph distance $d$ based on the shortest path of edges between vertices, and we consider closed balls with respect to this distance: 
\[ B(x,r) = \{ y \in V : d(x,y) \leq r \} \quad \forall x \in V, \ r >0.\]

We also assume $\Gamma$ has a random walk structure given by edge weights (conductances) $\mu_{xy}$ and vertex weights (measure) $\pi(x)$ with the following properties:
\begin{itemize}
\item $\mu_{xy} \not = 0 \iff \{x,y\} \in E$ and $\mu_{xy} = \mu_{yx}$
(the edge weights are \emph{adapted to the edges} and \emph{symmetric}) 
\item $\sum_{y \sim x} \mu_{xy} \leq \pi(x) \quad \forall x \in V$ 
(the edge weights are \emph{subordinate} to the measure/vertex weights).
\end{itemize}
The notation $y \sim x$ means that the unordered pair $\{ x, y \}$ belongs to the edge set $E.$ The notation $y \simeq x$ means either $y \sim x$ or $y=x.$ We will use $V(x,r)$ to denote the volume (with respect to $\pi$) of $B(x,r).$

Given a graph, we can always impose such a random walk structure on it; for example, we can take $\mu_{xy} \equiv 1 \ \forall \{x,y\} \in E$ and $\pi(x) = \deg(x) \ \forall x \in V.$ We will refer to this particular structure as \emph{simple weights}. 

Under the above assumptions, we define a Markov kernel $\mathcal{K}$ on $\Gamma$ via: 
\begin{equation}\label{MarkovKernel} 
\mathcal{K}(x,y) = \begin{cases} \displaystyle{\frac{\mu_{xy}}{\pi(x)}}, & x \not = y \\[2ex] \displaystyle{1 - \sum_{z \sim x} \frac{\mu_{xz}}{\pi(x)}} , & x=y.\end{cases}
\end{equation}
Hence loops are not allowed in $\Gamma,$ but the random walk is allowed to stay in place. 
Note that the Markov kernel $\mathcal{K}$ is \emph{reversible} with respect to the measure $\pi$, that is,
\[ \mathcal{K}(x,y) \pi(x) = \mathcal{K}(y,x) \pi(y) \quad \forall x, y \in \Gamma.\]
The random walk structure on $\Gamma$ may be equivalently defined by either a given $\mu, \pi,$ in which case $\mathcal{K}$ is as in (\ref{MarkovKernel}), or by a given Markov kernel $\mathcal{K}$ with reversible measure $\pi.$ 

Let $\mathcal{K}^n(x,y)$ denote the $n$-th convolution power of $\mathcal{K}(x,y).$ Then if $(X_n)_{n \geq 0}$ denotes the random walk on $\Gamma$ driven by $\mathcal{K}$, we have $\mathcal{K}^n(x,y) = \bP^x(X_n = y).$ The quantity $\mathcal{K}^n(x,y)$ is not symmetric (in particular, $\mathcal{K}$ itself need not be symmetric), so we will often be interested in studying instead its \emph{transition density}, the \emph{heat kernel} of the random walk, given by 
\[ p(n,x,y) = p_n(x,y) = \frac{\mathcal{K}^n(x,y)}{\pi(y)}. \]

There are various hypotheses one may make about the weights that have nice consequences. Here we will make the hypothesis of controlled weights.

\begin{definition}[Controlled Weights]
We say a graph $\Gamma$ has \emph{controlled weights} if there exists a constant $C_c > 1$ such that 
\begin{equation}\label{controlled} \frac{\mu_{xy}}{\pi(x)} \geq \frac{1}{C_c} \quad \forall x \in \Gamma, \ y \sim x .\end{equation}
\end{definition}

This assumption implies that $\Gamma$ is locally uniformly finite (that is, there is a uniform bound on the degree of any vertex) and that for $x \sim y,$ we have $\mu_{xy} \approx \pi(x) \approx \pi(y).$ \textbf{Unless stated otherwise, we will assume all graphs appearing have controlled weights.}

\subsection{Harnack Graphs}

In this section we describe several further properties graphs $(\Gamma, \pi, \mu)$ may possess and some of the consequences of these properties. 

\begin{definition}[Doubling]
A graph is said to be \emph{doubling} if there exists a constant $D$ such that for all $r>0, \ x \in \Gamma,$ 
\begin{equation}
V(x,2r) \leq D V(x,r).
\end{equation}
\end{definition}

\begin{definition}[Poincar\'{e} Inequality]
We say that $\Gamma$ satisfies the \emph{Poincar\'{e} inequality} if there exist constants $C_p >0, \ \kappa \geq 1$ such that for all $r > 0, x \in \Gamma,$ and functions $f$ supported in $B(x,\kappa r),$ 
\begin{equation*}
\sum_{y \in B(x,r)} |f(y) - f_B|^2 \, \pi(y) \leq C_p \ r^2 \sum_{y,z \in B(x, \kappa r)} |f(y)-f(z)|^2 \, \mu_{yz},
\end{equation*}
where $f_B$ is the (weighted) average of $f$ over the ball $B=B(x,r),$ that is,
\[ f_B = \frac{1}{V(x,r)} \sum_{y \in B(x,r)} f(y)\, \pi(y).\]
\end{definition}

Under doubling, the Poincar\'{e} inequality with constant $\kappa \geq 1$ (which appears in $B(x, \kappa r)$ on the right hand side) is equivalent to the Poincar\'{e} inequality with $\kappa = 1.$

\begin{definition}[Uniformly Lazy]
We say a pair $(\pi, \mu)$ is \emph{uniformly lazy} if there exists $C_e \in (0,1)$ such that 
\[ \sum_{y \sim x} \mu_{xy} \leq (1-C_e) \pi(x) \quad \forall x \in V, \ y \sim x.\]
We say a Markov kernel $\mathcal{K}$ is \emph{uniformly lazy} if there exists $C_e \in (0,1)$ such that 
\[ \mathcal{K}(x,x) = 1 - \sum_{z \sim x} \frac{\mu_{xz}}{\pi(x)} \geq  C_e \quad \forall x \in \Gamma. \]
\end{definition}

These two conditions are equivalent. In this case, the Markov chain is aperiodic. For instance, to turn a simple random walk ($\mu_{xy} \equiv 1$ and $\pi(x) = \text{deg}(x)$) into a lazy walk, just take $\mu_{xy} \equiv 1, \ \pi(x) = 2 \text{deg}(x).$ \textbf{Unless stated otherwise, we will consider all random walk structures appearing to be uniformly lazy.}

\begin{definition}[Harmonic Function]
A function $h: \Gamma \to \R$ is \emph{harmonic} (with respect to $\mathcal{K}$) if 
\begin{equation}\label{harmonic} h(x)  = \sum_{y \in \Gamma} \mathcal{K}(x,y) h(y) \quad \forall x \in \Gamma.\end{equation}
Given a subset $\Omega$ of $\Gamma$ (usually a ball), $h$ is harmonic on that set if (\ref{harmonic}) holds for all points in $\Omega$; this requires that $h$ be defined on $\{v\in \Gamma: \exists \omega\in \Omega, v\simeq \omega\}$. 
\end{definition}
As $\mathcal{K}(x,y) = 0$ unless $y \simeq x,$ the sum over $y \in \Gamma$ in (\ref{harmonic}) can be replaced by a sum over $y \simeq x.$ 

\begin{definition}[Elliptic Harnack Inequality]
A graph $(\Gamma, \pi, \mu)$ satisfies the \emph{ellliptic Harnack inequality} if there exist $\eta \in (0,1),\ C_H>0$ such that for all $r>0, \ x_0 \in \Gamma,$ and all non-negative harmonic functions $h$ on $B(x_0, r),$ we have 
\[ \sup_{B(x_0,\, \eta r)} h \leq C_H \inf_{B(x_0,\, \eta r)} h. \]
\end{definition}

\begin{definition}[Solution of Discrete Heat Equation]
A function $u: \Z_{+} \times \Gamma \to \R$ \emph{solves the discrete heat equation} if
\begin{equation}\label{heat_eqn} u(n+1, x) - u(n,x) = \sum_{y \in \Gamma} \mathcal{K}(x,y)[ u(n,y) - u(n,x)]  \quad \forall n \geq 1, \ x \in \Gamma.\end{equation}
Given a discrete space-time cylinder $Q = I \times B,$ $u$ solves the heat equation on $Q$ if (\ref{heat_eqn}) holds there (this requires that for each $n\in I$, $u(n,\cdot)$ is defined on $\{z\in \Gamma: \exists x \in B, z\simeq x\}$).\end{definition} 

\begin{definition}[Parabolic Harnack Inequality]
A graph $(\Gamma, \pi, \mu)$ satisfies the (discrete time and space) \emph{parabolic Harnack inequality} if: there exist $\eta \in (0,1), \ 0 < \theta_1 < \theta_2 < \theta_3 < \theta_4$ and $C_P>0$ such that for all $s,r >0, \ x_0 \in \Gamma,$ and every non-negative solution $u$ of the heat equation in the cylinder $Q = (\Z_{+} \cap [s, s+ \theta_4 r^2]) \times B(x_0,r),$ we have
\[ u(n_-, x_-) \leq C_P \, u(n_+, x_+) \quad \forall (n_-, x_-) \in Q_-, \ (n_+, x_+) \in Q_+ \text{ s.t. } d(x_-, x_+) \leq n_+ - n_- ,\]
where $Q_- = (\Z_+ \cap [ s + \theta_1 r^2, s + \theta_2 r^2]) \times B(x_0, \eta r)$ and $Q_+ = (\Z_+ \cap [ s+ \theta_3 r^2, s+ \theta_4 r^2]) \times B(x_0, \eta r).$ 
\end{definition}

The parabolic Harnack inequality obviously implies the elliptic version. The following theorem relates several of the above concepts. 

\begin{theorem}[Theorem 1.7 in \cite{Delmotte_PHI}]\label{VD_PI}
Given $(\Gamma, \pi, \mu)$  (or $(\Gamma, \mathcal{K}, \pi)$) where $\Gamma$ has controlled weights and $\mathcal K$ is uniformly lazy, the following are equivalent:
\begin{enumerate}[(a)]
\item $\Gamma$ is  doubling and satisfies the Poincar\'{e} inequality 
\item $\Gamma$ satisfies the parabolic Harnack inequality
\item $\Gamma$ satisfies two-sided Gaussian heat kernel estimates, that is there exists constants $c_1, c_2, c_3, c_4 >0$ such that for all $x,y \in V, \ n \geq d(x,y),$
\begin{equation}\label{2sided_graphs}
\frac{c_1 }{V(x, \sqrt{n})} \exp\Big(-\frac{d^2(x,y)}{c_2 n}\Big) \leq p(n,x,y) \leq \frac{c_3}{V(x, \sqrt{n})} \exp\Big(-\frac{d^2(x,y)}{c_4 n}\Big).
\end{equation}
\end{enumerate}
\end{theorem}

\begin{definition}[Harnack Graph]
If $(\Gamma, \pi, \mu)$ satisfies any of the conditions in Theorem \ref{VD_PI}, we call $\Gamma$ a \emph{Harnack graph}. 
\end{definition}

\begin{remark}
The uniformly lazy assumption avoids problems related to parity (such as those that appear in bipartite graphs). Without this assumption, it may be that (a) holds but $p(n,x,y) = 0$ for some $n \geq d(x,y),$ and then (b) and the lower bound in (c) do not hold. Here we avoid such difficulties by assuming the graph is uniformly lazy; another solution to this problem is to state (b) and (c) for the sum over two consecutive discrete times $n,\ n+1$, e.g., for (c), $p(n,x,y) + p(n+1,x,y).$ 
\end{remark}

\begin{definition}[The notation $\approx$]\label{approx_def}
For two functions of a variable $x,$ the notation $f \approx g$ means there exist constants $c_1, c_2$ (independent of $x$) such that 
\[ c_1 f(x) \leq g(x) \leq c_2 g(x).\]
\end{definition}

\begin{definition}[Abuse of $\approx$]\label{approx_def2} 
We will often abuse the notation $\approx$ in the case of heat kernel and hitting probability estimates to write formulas more compactly. For instance, we will write (\ref{2sided_graphs}) as 
\[ p(n,x,y) \approx \frac{1}{V(x, \sqrt{n})} \exp\Big(-\frac{d^2(x,y)}{n}\Big).\]
Note this use of $\approx$ means there are different constants in the upper and lower bounds both inside \emph{and} outside the exponential. In the event that we chain such notations together, all constants may change from line to line. 
\end{definition} 

\subsection{Subgraphs of a larger graph}\label{subgraph}

Sometimes we think of $\Gamma$ as a subgraph of a larger graph $\widehat{\Gamma} =(\widehat{V}, \widehat{E}).$ If given $\widehat{\Gamma},$ then for any subset of $V$ of $\widehat{V},$ we can construct a graph $\Gamma$ with vertex set $V$ and edge set $E$ where $\{ x , y \} \in E$ if and only if $x,y \in V$ and $\{x,y\} \in \widehat{E}.$ On occasion, we will abuse notation and use the same symbol to denote both a subset of $\widehat{V}$ and its associated subgraph. 

Further, a subgraph $\Gamma$ inherits a random walk structure from $\widehat{\Gamma}.$ We set $\pi_{\Gamma}(x) = \pi_{\widehat{\Gamma}}(x)$ and $\mu^{\Gamma}_{xy} = \mu^{\widehat{\Gamma}}_{xy}$ for all $x,y \in V, \ \{x,y\} \in E.$ (Hence we may simply use $\pi, \mu$ without indicating the whole graph versus the subgraph, provided that $x,y \in \Gamma.$) 

Then we may define a Markov kernel on $\Gamma$ as in (\ref{MarkovKernel}); we call ths Markov kernel the Neumann Markov kernel of $\Gamma$ (with respect to $\widehat{\Gamma}$) and denote it by $\mathcal{K}_{\Gamma, N}.$ To be precise,

\begin{equation}
\mathcal{K}_{\Gamma,N}(x,y) = \begin{cases} \displaystyle{\frac{\mu^\Gamma_{xy}}{\pi(x)} = \frac{\mu_{xy}^{\widehat{\Gamma}}}{\pi(x)},} & x \not = y, \ x,y \in V \\[2ex]  \displaystyle{1 - \sum_{z \sim x} \frac{\mu^{\Gamma}_{xz}}{\pi(x)} = 1 - \sum_{z \sim x, z \in V} \frac{\mu^{\widehat{\Gamma}}_{xz}}{\pi(x)}}, & x=y \in V.\end{cases}
\end{equation}

We can also define the Dirichlet Markov kernel of $\Gamma$ (with respect to $\widehat{\Gamma}$) by
\begin{equation}
\mathcal{K}_{\Gamma,D}(x,y) = \mathcal{K}_{\widehat{\Gamma}}(x,y) \mathds{1}_{V}(x) \mathds{1}_{V}(y) = \begin{cases} \frac{\mu^\Gamma_{xy}}{\pi(x)}, & x \not = y, \ x, y \in V \\ 1 - \sum_{z \sim x, z \in \widehat{V}} \frac{\mu^{\widehat{\Gamma}}_{xz}}{\pi(x)} , & x=y \in V,\end{cases}
\end{equation}
where $\mathds{1}_V(x) = 1$ if $x \in V$ and zero otherwise. When $V \not = \widehat{V}, \ \mathcal{K}_{\Gamma, D}$ is only a subMarkovian kernel. 

A subgraph $\Gamma$ comes with its own notion of distance $d_\Gamma,$ where $d_\Gamma(x,y)$ is the length of the shortest path between $x$ and $y$ of edges contained in $\Gamma.$ It is always true that $d_{\widehat{\Gamma}}(x,y) \leq d_\Gamma(x,y).$ 

There are two natural notions for the boundary of $\Gamma,$ both of which are useful to us here. 

\begin{definition}[Exterior/Inner Boundary]
The \emph{(exterior) boundary} of $\Gamma$ is 
\[ \partial \Gamma = \{ y \in \widehat{\Gamma} \setminus \Gamma: \exists x \in \Gamma \text{ s.t. } d_{\widehat{\Gamma}}(x,y) = 1\},\]
in other words, the set of points that do not belong to $\Gamma$ with neighbors in $\Gamma.$

The \emph{inner boundary} of $\Gamma$ is the set of points inside $\Gamma$ with neighbors outside of $\Gamma,$ 
\[ \partial_I \Gamma = \{ x \in \Gamma: \exists y \not \in \Gamma \text{ s.t. } d_{\widehat{\Gamma}}(x,y) = 1\}.\]
\end{definition} 

For $x \in \Gamma$ and $y \in \partial \Gamma$ ($\not \in \Gamma$), we extend the definition of $d_\Gamma(x, \cdot)$ by setting 
\[ d_\Gamma(x,y) = 1 + \min_{z \in \Gamma: z \sim y} d_{\Gamma}(z,x).\]
This extension is \emph{not} a distance function as it need not satisfy the triangle inequality. The correct way to think of adding points in $\partial \Gamma$ to $\Gamma$ is to think of them as multiple points as described in Figure \ref{bdry}. If the boundary points are duplicated appropriately, this extension can be made into a distance function. 

\begin{figure}[]
\centering
\begin{tikzpicture}[scale=.45]

       \foreach \i in {0,...,10}
        \draw[gray] (\i, 0)--(\i,10);
       \foreach \j in {0,...,10}
        \draw[gray] (0,\j)--(10,\j);
              \foreach \i in {0,...,10}
      \foreach \j in {0,...,10}
        \filldraw (\i,\j) circle(4pt);
      \foreach \j in {0,...,7}
        \filldraw[red] (5,\j) circle(5pt);
       \foreach \j in {0,...,7}
        \filldraw[blue] (4,\j) circle(5pt);
        \foreach \j in {0,...,7}
        \filldraw[blue] (6,\j) circle(5pt);
        \filldraw[blue] (5,8) circle(5pt);
        \filldraw[SpringGreen3] (1,2) circle(6pt);
         \filldraw[DarkGoldenrod2] (8,2) circle(6pt);
         \node at (0.5,1.5) {$x$};
         \node at (8.5,1.5) {$y$};
\end{tikzpicture}\hphantom{spacespace}
\begin{tikzpicture}[scale=.45]
        \foreach \i in {0,...,4,9,10,11,12,13}
        \draw[gray] (\i, 0)--(\i,10);
        \foreach \j in {0,...,10}
        \draw[gray] (0,\j)--(4,\j);
       \foreach \j in {0,...,10}
        \draw[gray] (9,\j)--(13,\j);
        \draw[gray] (4,10)--(9,10);
        \draw[gray] (4,9)--(9,9);
         \draw[gray] (4,8)--(9,8);
         \draw[gray](5,0)--(5+7/6,7);
          \draw[gray](8,0)--(8-7/6,7);
          \draw[gray](5+7/6,7)--(6.5,7.5);
          \draw[gray](8-7/6,7)--(6.5,7.5);
      \foreach \k in {0,...,7}
        \draw[gray] (4, \k)--(5+ \k/6, \k);
       \foreach \k in {0,...,7}
        \draw[gray] (9, \k)--(8-\k/6, \k);
        \draw[gray] (6.5,8)--(6.5,7.5);
     \foreach \i in {0,1,2,3,4,9,10,11,12,13}
      \foreach \j in {0,...,10}
        \filldraw (\i,\j) circle(4pt);
     \foreach \k in {0,...,7}
       \filldraw[red] (5 + \k/6, \k) circle(4pt);
      \foreach \k in {0,...,7}
       \filldraw[orange] (8 - \k/6, \k) circle(4pt);
       \filldraw[Coral3] (6.5,7.5) circle(4pt); 
       \foreach \j in {0,...,7}
        \filldraw[blue] (4,\j) circle(4pt);
        \foreach \j in {0,...,7}
         \filldraw[blue] (9,\j) circle(5pt);
         \filldraw (6.5, 10) circle(4pt);
         \filldraw (6.5, 9) circle(4pt);
          \filldraw[blue] (6.5, 8) circle(4pt);
        \filldraw[SpringGreen3] (1,2) circle(6pt);
         \filldraw[DarkGoldenrod2] (11,2) circle(6pt);
         \node at (0.5,1.5) {$x$};
         \node at (11.5,1.5) {$y$};
\end{tikzpicture}
 \caption{Let $\widehat{\Gamma}$ be the full ten edges by ten edges square as on the left. Take $\Gamma$ to be $\widehat{\Gamma}$ minus the red points. The red points are $\partial \Gamma$, and the blue points are $\partial_I \Gamma$. Then $d(x, \partial \Gamma) = 4$ and $d(y, \partial \Gamma) = 3,$ and both of these distances are achieved by the same point in $\partial \Gamma,$ call it $z.$ Note $d_\Gamma(x,y) = 19 > d_{\Gamma}(x,z) + d_\Gamma(y,z) = 7.$ The correct way to think of this is duplicating the red points of $\partial \Gamma$ as shown in the right figure.}\label{bdry}
  \end{figure}
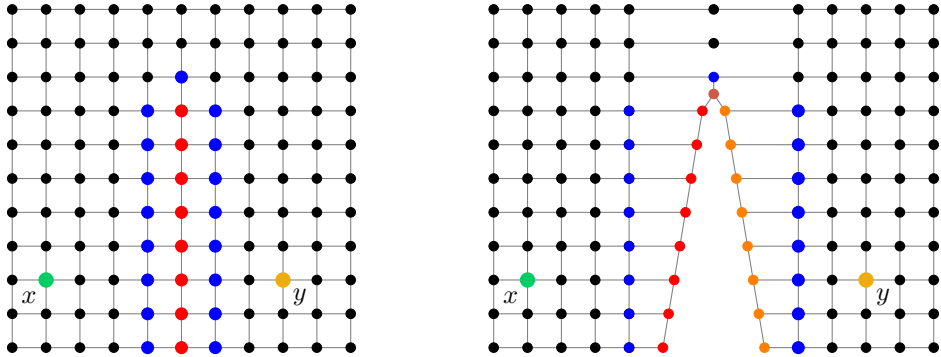

\section{Hitting Probabilities and \texorpdfstring{$S$}{S}-transient Graphs}\label{upper}

\subsection{Hitting Probability Upper Bound}

For this section, consider a graph $\widehat{\Gamma} = (\widehat{V}, \widehat{E})$ with controlled and uniformly lazy weights $(\mu, \pi)$. Let $K$ be a subset of $\widehat{V},$ where we abuse notation to let $K$ indicate both this set of vertices and the subgraph of $\widehat{\Gamma}$ induced by these vertices. Set $\Gamma := \widehat{\Gamma} \setminus K,$ that is, we think of $\Gamma$ as the subgraph of $\widehat{\Gamma}$ with vertex set $\widehat{V} \setminus K.$ We are interested in transience properties of $\Gamma$ and the hitting probability of $K.$ We will assume $\widehat{\Gamma}, \ \Gamma$ to be infinite and connected; $K$ may be either finite or infinite and connected or disconnected. 

We are used to thinking of Markov (or subMarkovian) kernels as recurrent if random walks return to their starting points infinitely often and transient if they do not. However, in the present setting of a subgraph which inherits a random walk structure from a larger graph, there are several natural ways to think of transience/recurrence. 

\begin{definition}[$N$-transience]
A subgraph $\Gamma$ of $(\widehat{\Gamma}, \mathcal{K}, \pi)$ is $N$-transient (``Neumann"-transient) if $(\Gamma, \mathcal{K}_{\Gamma, N}, \pi)$ is transient, that is, with probability one, a random walk on $\Gamma$ only returns to its starting point finitely often.
\end{definition}

Being $N$-transient is an intrinsic property of the subgraph $\Gamma,$ which is in some sense independent of $\widehat{\Gamma}$. A similar definition could be given using $(\Gamma, \mathcal{K}_{\Gamma, D}, \pi)$ instead to define $D$-transience. The killing present in the subMarkovian kernel $\mathcal{K}_{\Gamma,D}$ makes $D$-transience more likely. 

However, in this paper, the main definition of transience we will be concerned with is $S$-transience, or ``survival"-transience, defined below. The explanation for the name is that a subgraph of a larger graph is survival-transient if there is positive probability that a random walk started inside the subgraph never sees vertices that do not belong to the subgraph, hence survives forever.

\begin{definition}[Hitting Time, Hitting Probability]
Consider a graph $(\widehat{\Gamma}, \mathcal{K}, \pi)$ with random walk denoted by $(X_n)_{n \geq 0}.$ Then we denote the first hitting time of a subset of vertices $K$ by  $\ \tau_{K} := \min \{ n \geq 0 : X_n \in K\}$ and the first return time to $K$ by $\tau_K^+ := \min \{ n \geq 1 : X_n \in K\}.$ If $X_0 \not \in K,$ then $\tau_K$ and $\tau_K^+$ are the same. Further, denote the hitting probability of $K$ by $\psi_{K}(x) = \bP^x(\tau_K < + \infty).$
\end{definition}

\begin{definition}[$S$-transience]
Let $(\widehat{\Gamma}, \mathcal{K}, \pi)$ be a connected graph with controlled weights and $K$ be a subset of $\widehat{\Gamma}$ such that $\Gamma := \widehat{\Gamma} \setminus K$ is connected. We say the subgraph $\Gamma$ is \emph{$S$-transient} (``survival"-transient) or that the graph \emph{$\widehat{\Gamma}$ is $S$-transient with respect to the set $K$} if there exists $x \in \widehat{\Gamma}$ such that $\psi_K(x) < 1.$ 

If this is not the case, then we say $\Gamma$ is \emph{$S$-recurrent} (or that $\widehat{\Gamma}$ is $S$-recurrent with respect to $K$).
\end{definition}

\begin{lemma}[Equivalent Definitions of $S$-transience]
Let $(\widehat{\Gamma}, \mathcal{K}, \pi)$ be a connected graph with controlled weights and $K$ be a subset of $\widehat{\Gamma}$ such that $\Gamma := \widehat{\Gamma} \setminus K$ is connected. Then the following are equivalent
\begin{enumerate}[(a)]
\item There exists $x \in \widehat{\Gamma}$ such that $\psi_K(x) < 1.$ 
\item For all $x \in \Gamma, \ \psi_K(x) < 1.$ 
\item For all $y \in \partial \Gamma, \ \bP^y(\tau_K^+ < + \infty) <1 .$ 
\end{enumerate}
\end{lemma}

\begin{proof}
Clearly (b) implies (a). That (a) implies (b) follows from the maximum principle: By the definition of a hitting probability, $\psi_K$ is a harmonic function on $\Gamma := \widehat{\Gamma} \setminus K;$ thus so is $1 - \psi_K,$ which is non-negative. By the maximum principle, if there exists some $x \in \Gamma$ such that $1- \psi_K(x) =0,$ then $1-\psi_K \equiv 0$ on $\Gamma.$ Hence if $\psi_K(x) <1$ for a single $x \in \Gamma,$ this must be true of all $x \in \Gamma.$

We now show the equivalence of (a)-(b) and (c).  If $y \in \partial \Gamma,$ then using the Markov property, 
\begin{align*}
\bP^y(\tau_K^+ < + \infty) &= \sum_{x \simeq y} \bP^y (\tau_K^+ < + \infty, X_1 = x) = \sum_{x \simeq y} \bE^y ( \mathds{1}_{\{X_1 = x\}} \bE^{X_1}(\mathds{1}_{\{\tau_K < + \infty\}})) \\&= \sum_{x \simeq y} \mathcal{K}(y,x) \bP^x(\tau_K  < + \infty) = \sum_{x \simeq y} \psi_K(x) \mathcal{K}(y,x). 
\end{align*}

Since $y \in \partial \Gamma,$ there exists $z \in \Gamma$ such that $z \sim y.$ If (b) holds, then $\psi_K(z) < 1$ so 
\[ \bP^y (\tau_K^+ < + \infty) < \sum_{x \simeq y} \mathcal{K}(y,x) = 1\]
and (c) holds. Conversely, if (c) holds, then 
\[ \sum_{x \simeq y} \psi_K(x) \mathcal{K}(x,y) < 1 = \sum_{x \simeq y} \mathcal{K}(x,y).\]
Thus there exists some $z \sim y$ such that $\psi_K(z) < 1,$ so (a) holds. 
\end{proof}

Note the lemma does not contain some of the other usual equivalent definitions of transience as allowing for $K$ to be infinite can cause difficulties. For example, whether $K$ is hit infinitely often or not can depend upon the precise choice of $K$ as well as upon if the walk starts inside or outside of $K.$ A graph is transient in the classical sense if and only if it is $S$-transient with respect to any finite set.

\begin{example}[Lattices $\Z^m$]
The lattice $\Z^m$ with simple weights is (classically) transient or, equivalently, $S$-transient with respect to any finite set, if and only if $m\geq 3.$ 
\end{example} 

\begin{example}[Lattices in lattices, $\Z^m \setminus \Z^k$]
Consider a copy of a $k$-dimensional lattice $\Z^k$ inside of $\Z^m,$ where we assume $k\leq m$ and $m\geq 3.$

 If $k \leq m-3,$ then $\Z^m \setminus \Z^k$ is connected, $N$-transient, and $S$-transient.
 
 If $k = m-2,$ then $\Z^m \setminus \Z^k$ is connected and $N$-transient, but it is not $S$-transient, since the set $\Z^k$ will be visited infinitely often with probability one (and hence certainly $\psi_{\Z^k} \equiv 1$).

 If $k = m-1,$ then $\Z^m \setminus \Z^k$ is disconnects into two half-spaces (see Example \ref{half_space} below). 
\end{example}

\begin{example}[Half-space $\Z^m_+$]\label{half_space}
Consider the upper half-space $\Gamma = \Z^m_+ = \{ (x_1, \dots, x_{m-1}, \allowbreak x_m) \in \Z^m : x_m > 0\}$
inside of $\Z^m$ with simple weights. Then $\Z^m_+$ is $N$-transient if and only if $m \geq 3,$ and it is always transient if we kill the walk along the set $\{x_m=0\}.$ However, $\Z^m_+$ is \emph{never} $S$-transient, since as the walk on $\Z^m_+$ escapes to infinity, it hits the set $\{x_m = 0\}$ with probability one. 
\end{example}

\begin{definition}[Uniform $S$-transience]
Let $(\widehat{\Gamma}, \pi, \mu)$ be a graph and $\Gamma := \widehat{\Gamma} \setminus K$ a subgraph. We say $\Gamma$ is \emph{uniformly $S$-transient}, or that $\widehat{\Gamma}$ is \emph{uniformly $S$-transient with respect to} $K$, if there exist $L, \varepsilon >0$ such that $d(x,K) \geq L$ implies that $\psi_K(x) \leq 1- \varepsilon.$ 
\end{definition}

The following theorem gives an upper bound on the hitting probability of $K.$ This bound can be useful for showing $S$-transience. 

\begin{theorem}[Hitting Probability Upper Bound]\label{psi-upperbd}
Let $(\widehat{\Gamma}, \pi, \mu)$ have controlled weights that are uniformly lazy.  Let $K$ be a subset/subgraph of $\widehat{\Gamma}.$ Set $\Gamma := \widehat{\Gamma} \setminus K$ and note $\partial \Gamma \subseteq K.$ Assume that $\widehat{\Gamma}$ is Harnack. Define
\[ B_{\partial \Gamma}(x,r) = B_{\widehat{\Gamma}}(x,r) \cap \partial \Gamma \quad \text {and} \quad V_{\partial \Gamma}(x,r) = \pi(B_{\partial \Gamma}(x,r)) \quad \forall x \in \widehat{\Gamma},\]
that is, $V_{\partial \Gamma}$ is the volume of traces of $\widehat{\Gamma}$-balls in $\partial \Gamma.$

For any $x \in \Gamma = \widehat{\Gamma} \setminus K,$ set
\[ W(x,r) := \frac{V_{\widehat{\Gamma}}(x,r)}{V_{\partial \Gamma}(x,r)}.\]
Then, if $d_x := d(x, K),$ there exists a constant $C$ (depending on the constants appearing in the controlled weights, uniformly lazy, and Harnack assumptions) such that 
\begin{equation}\label{psi_bd}
\psi_K(x) \leq \sum_{n \geq d_x^2} \frac{C}{W(x, \sqrt{n})} \quad \forall x \in \Gamma \setminus \partial_I \Gamma.
\end{equation}
\end{theorem}

The theorem does not discuss $x \in \partial_I \Gamma,$ since in this case $\psi_K(x) \approx 1,$ independently of $x,$ due to the controlled weights hypothesis. 

Using the theorem, it is easy to verify that $\Z^m \setminus \Z^k$ is uniformly $S$-transient when $k \leq m-3$ (see Example \ref{lattice_lattice} below).

\begin{proof}
For any $x \in \Gamma \setminus \partial_I \Gamma, \ d(x, K) \geq 2,$ we have
\begin{equation}\label{psi_equals} 
\begin{aligned}
\psi_{K}(x) &:= \bP^x( \tau_{K} < + \infty)
= \sum_{n=1}^\infty \bP^x(\tau_K = n) = \sum_{n \geq 1} \sum_{v \in K} \bP^x (\tau_K = n, X_{\tau_K} = v) \\
&= \sum_{n \geq 2} \sum_{v \in \partial \Gamma} \sum_{\substack{y \sim v \\ y \in \Gamma}} \mathcal{K}_{\Gamma,D}^{n-1}(x,y) \mathcal{K}_{\widehat{\Gamma}}(y,v).
\end{aligned}
\end{equation}

Since $\mathcal{K}_{\widehat{\Gamma}}(y,v)$ is a probability, it is at most one. The Dirichlet Markov kernel on $\Gamma$ is less than the Markov kernel on all of $\widehat{\Gamma},$ which is Harnack. Hence
\begin{align*}
\psi_K(x) &= \sum_{n \geq 2} \sum_{v \in \partial \Gamma} \sum_{\substack{y \sim v \\ y \in \Gamma}} \mathcal{K}_{\Gamma,D}^{n-1}(x,y) \mathcal{K}_{\widehat{\Gamma}}(y,v)
\leq \sum_{n \geq 2} \sum_{v \in \partial \Gamma} \sum_{\substack{y \sim v \\ y \in \Gamma}} \mathcal{K}_{\widehat{\Gamma}}^{n-1}(x,y) \\
&\leq C \sum_{n \geq 2} \sum_{v \in \partial \Gamma} \sum_{\substack{y \sim v \\ y \in \Gamma}} \frac{\pi(y)}{V_{\widehat{\Gamma}}(x, \sqrt{n})} \exp \Big(- \frac{d_{\widehat{\Gamma}}^2(x,y)}{cn}\Big).
\end{align*}

Now $\pi(y) \approx \pi(v)$ and the number of $y \sim v$ is uniformly bounded above. Moreover, as any of the above $y'$s belong to  $\partial_I \Gamma$ while $x$ does not, it is impossible that $x = y.$ Hence $d_{\widehat{\Gamma}}(x,y) \geq 1.$ Thus we can replace $d_{\widehat{\Gamma}}(x,y)$ via $d_{\widehat{\Gamma}}(x,v)$ since 
\[ d_{\widehat{\Gamma}}(x,v) \leq d_{\widehat{\Gamma}}(x,y) + 1 \leq 2 d_{\widehat{\Gamma}}(x,y).\]

We have 
\begin{align*}
\psi_K(x) &\leq C \sum_{n \geq 2} \sum_{v \in \partial \Gamma} \frac{\pi(v)}{V_{\widehat{\Gamma}}(x, \sqrt{n})} \exp \Big(- \frac{d_{\widehat{\Gamma}}^2(x,v)}{cn}\Big).
\end{align*} 

Let $d := d_{\widehat{\Gamma}}(x,v).$ We sum first in time. First we split the sum into two, noting that the exponential is large for large $n$; recall the notation $\approx$ is as in Definition \ref{approx_def2} and means we have matching upper/lower bounds with different constants inside and outside the exponential:
\begin{align*}
\sum_{n \geq 2} \frac{1}{V_{\widehat{\Gamma}}(x, \sqrt{n})} \exp\Big(-\frac{d_{\widehat{\Gamma}}^2(x,v)}{n}\Big)
&\approx \sum_{n \leq d^2} \frac{1}{V_{\widehat{\Gamma}}(x, \sqrt{n})} \exp\Big(-\frac{d_{\widehat{\Gamma}}^2(x,v)}{n}\Big) + \sum_{n \geq d^2} \frac{1}{V_{\widehat{\Gamma}}(x, \sqrt{n})}.
\end{align*}

We compute the first piece of the sum by arranging a dyadic decomposition with $d/\sqrt{n} \asymp 2^l,$ where the notation $\asymp$ means $2^{l} \leq d/\sqrt{n} \leq 2^{l+1}.$ The quantity $d/\sqrt{n}$ ranges from $1$ to $d$ here. Let $l_{x,v}$ be the integer such that $d(x,v) \asymp 2^{l_{x,v}}.$ Then, with constants $C,c$ changing from one inequality to the next,
\begin{align*}
\sum_{n \leq d^2} \frac{C}{V_{\widehat{\Gamma}}(x, \sqrt{n})} \exp\Big( - \frac{d^2_{\widehat{\Gamma}}(x, v)}{cn}\Big)
&\leq \sum_{l =0}^{l_{x,v}} \sum_{\sqrt{n} \asymp d 2^{-l}} \frac{C}{V_{\widehat{\Gamma}}(x, \sqrt{n})} \exp\Big( - \frac{d^2}{cn}\Big)\\
&\leq \sum_{l=0}^{l_{x,v}} \frac{C}{V_{\widehat{\Gamma}}(x, d/2^{l+1})} \frac{d^2}{4^l} \exp\Big(-\frac{4^l}{c}\Big) \\
&\leq C \frac{d^2}{V_{\widehat{\Gamma}}(x,d)} \sum_{l=0}^{l_{x,v}} \exp\Big(-\frac{4^l}{c}\Big) \leq  C \frac{d^2}{V_{\widehat{\Gamma}}(x,d)}.
\end{align*}
In the last line we used the doubling of $\widehat{\Gamma}$ (a consequence of the assumption that $\widehat{\Gamma}$ is Harnack). 

It is easy to see the sum we just computed is dominated by the tail sum ($n \geq d_{\widehat{\Gamma}}^2(x,v)$) as 
\[ \sum_{n=d^2}^{4d^2} \frac{1}{V_{\widehat{\Gamma}}(x, \sqrt{n})} \approx \frac{d^2}{V_{\widehat{\Gamma}}(x,d)} \]
due to the doubling of $\widehat{\Gamma}.$ Recall here the exponential is large. 

Let $d_x := d(x,K)$ and note $d(x,K) \leq d(x,v)$ for all $v \in \partial \Gamma$. Switching the order of summation in the upper bound above, we find
\begin{align*}
\psi_K(x) &\leq C \sum_{v \in \partial \Gamma} \sum_{n \geq d^2(x,v)} \frac{\pi(v)}{V_{\widehat{\Gamma}}(x,\sqrt{n})}
= \sum_{n \geq d_x^2} \sum_{\substack{ v \in \partial \Gamma: \\ d^2(x,v) \leq n}} \frac{\pi(v)}{V_{\widehat{\Gamma}}(x,\sqrt{n})} \\
&= C\sum_{n \geq d_x^2} \frac{V_{\partial \Gamma}(x, \sqrt{n})}{V_{\widehat{\Gamma}}(x, \sqrt{n})} 
= \sum_{n \geq d_x^2} \frac{C}{W(x, \sqrt{n})}. 
\end{align*}

\end{proof}

In the case the volume of traces of $\widehat{\Gamma}$-balls in $\partial \Gamma$ are doubling, the theorem simplifies. It is important that in the corollary we only consider traces whose centers belong to $\partial \Gamma.$ 

\begin{definition}[Doubling Boundary]\label{bdry_double}
Consider a graph $(\widehat{\Gamma}, \pi, \mu)$ and subgraph $\Gamma$ of $\widehat{\Gamma}$ with exterior boundary $\partial \Gamma.$ We say $V_{\partial \Gamma}$ is doubling if there exists a constant $D>0$ such that for all $z \in \partial \Gamma, \ r>0, \ V_{\partial \Gamma}(z,2r) \leq D V_{\partial \Gamma}(z,r).$ 
\end{definition}

\begin{corollary}\label{simpler} 

Under the assumptions of Theorem \ref{psi-upperbd} and the additional assumption that $V_{\partial \Gamma}$ is doubling (as in Definition \ref{bdry_double}), then the upper bound of Theorem \ref{psi-upperbd} has the following simplified form. Let $v_x \in \partial \Gamma$ achieve $d(x,K)$ and set $\widetilde{W}(x,r) := V_{\widehat{\Gamma}}(x,r)/V_{\partial \Gamma}(v_x,r).$ 

Then there exists a constant $C$ (depending on $D$ from Definition \ref{bdry_double} and the constants from the assumptions as in Theorem \ref{psi-upperbd}) such that
\[ \psi_K(x) \leq \sum_{n \geq d_x^2} \frac{C}{\widetilde{W}(x, \sqrt{n})}.\]
\end{corollary}

\begin{proof}
Return to the point in the proof of Theorem \ref{psi-upperbd} where
\begin{align*}
\psi_K(x) &\leq C \sum_{v \in \partial \Gamma} \sum_{n \geq d_{\widehat{\Gamma}}^2(x,v)} \frac{\pi(v)}{V_{\widehat{\Gamma}}(x,\sqrt{n})}.
\end{align*}

Then $d_{\widehat{\Gamma}}(x, v_x) \leq d_{\widehat{\Gamma}}(x, v)$ by definition of $v_x$ and $d_{\widehat{\Gamma}}(v_x,  v) \leq d_{\widehat{\Gamma}}(v_x, x) + d_{\widehat{\Gamma}}(x, v) \leq 2d_{\widehat{\Gamma}}(x, v).$ Therefore $d_{\widehat{\Gamma}}(v,x) \geq \frac{1}{3}(d_{\widehat{\Gamma}}(x, v_x) + d_{\widehat{\Gamma}}(v_x,v)),$ and we can replace $d_{\widehat{\Gamma}}(v,x)$ with this sum, that is,
\begin{align*}
\psi_K(x) &\leq C \sum_{v \in \partial \Gamma} \sum_{n \geq \frac{1}{9} (d_{\widehat{\Gamma}}^2(x,v_x)+d_{\widehat{\Gamma}}^2(v_x,v))} \frac{\pi(v)}{V_{\widehat{\Gamma}}(x,\sqrt{n})}.
\end{align*}

Again interchanging the order of summation, noting that the time sum for a particular $v$ requires $n \geq \frac{1}{9} d_{\widehat{\Gamma}}^2(v_x,v),$ we have 
\begin{align*}
\psi_K(x) &\leq C \sum_{n \geq \frac{1}{9}d_{\widehat{\Gamma}}^2(x, v_x)} \frac{1}{V_{\widehat{\Gamma}}(x, \sqrt{n})} \sum_{\substack{v \in \partial \Gamma: \\ \frac{1}{9} d_{\widehat{\Gamma}}^2(v_x, v) \leq n}} \pi(v) \\
&\leq C \sum_{n \geq d_{\widehat{\Gamma}}^2(x, v_x)} \frac{V_{\partial \Gamma}(v_x, 3 \sqrt{n})}{V_{\widehat{\Gamma}}(x, \sqrt{n})}
\leq C \sum_{n \geq d_{\widehat{\Gamma}}^2(x, v_x)} \frac{V_{\partial \Gamma}(v_x, \sqrt{n})}{V_{\widehat{\Gamma}}(x, \sqrt{n})},
\end{align*}
where we used the doubling of both $V_{\partial \Gamma}$ and $V_{\widehat{\Gamma}}$ in the last line. 
\end{proof}

\begin{remark}
If $\{ \sum_{n \geq 1} \big(W(x, \sqrt{n})\big)^{-1} \}_{x \in \Gamma}$ (or the sum with $\widetilde{W}$) converges uniformly, that is, for all $\varepsilon>0,$ there exists $N$ (independent of $x$) such that 
\begin{equation}\label{W_unif}
\sum_{n \geq N} \frac{1}{W(x,\sqrt{n})} < \varepsilon,
\end{equation}
then it follows from Theorem \ref{psi-upperbd} that $\widehat{\Gamma}$ is uniformly $S$-transient with respect to $K.$ In fact, in this case, $\psi_K(x) \to 0$ uniformly as $d_{\widehat{\Gamma}}(x,K) \to \infty.$
\end{remark}

In the $S$-recurrent case, $\sum_{n\geq 1} \big(W(x, \sqrt{n})\big)^{-1}$ need not converge. We will also see examples where this sum converges, but not uniformly in $x$. In certain regimes of the latter type of example, we may be able to use (\ref{psi_bd}) to see that $\psi_K < 1$ for some $x$ (and hence show $S$-transience). In other regimes and in the recurrent case, this bound is not useful. (Recall $\psi(x) \leq 1 \ \forall x$ since $\psi$ is a probability.) Observe that the above theorem is not strong enough to conclude that if $\sum_{n \geq 1} \big(W(x, \sqrt{n})\big)^{-1}$ converges for some (any) $x \in \widehat{\Gamma} \setminus K$ then $\widehat{\Gamma}$ is $S$-transient with respect to $K.$ This is because in the bound given by the theorem, the start of tail of the sum depends on the point $x,$ so although $\sum_{n\geq 1} \big(W(x, \sqrt{n})\big)^{-1}$ may converge, that does not guarantee that $\sum_{n \geq d_x^2} \big(W(x, \sqrt{n})\big)^{-1}$ is sufficiently small to make $\psi_K(x) < 1.$

\subsection{Examples}

In this section we give examples of applying Theorem \ref{psi-upperbd} or Corollary \ref{simpler} to demonstrate $S$-transience or uniform $S$-transience. Below we frequently use $\approx$ from Definition \ref{approx_def}.

\begin{example}[Lattices in lattices]\label{lattice_lattice} We verify that $\Z^m$ is uniformly $S$-transient with respect to $\Z^k$ whenever $m-k \geq 3.$ 

Consider $\widehat{\Gamma} = \Z^m$ with $K = \partial \Gamma = \Z^k = \{ (x_1, \dots, x_k, 0, \dots, 0) \in \Z^m : x_1, \dots, x_k \in \Z\}$ the $k$-dimensional sublattice made up of the first $k$-coordinates. Assume $m-k \geq 3.$ Suppose $\Z^m$ has simple weights or a variation thereof (such as taking bounded weights or taking the lazy SRW on $\Z^m$). With these weights, $\Z^m$ is Harnack, and $B_{\partial \Gamma}(z,r) = B_{\Z^m}(z,r) \cap \Z^k = B_{\Z^k}(z,r)$ for all $z \in \Z^k,$ so it is clear $V_{\Z^k}$ is doubling. Thus all the hypotheses of Theorem \ref{psi-upperbd} and Corollary \ref{simpler} are satisfied.
 We compute
\[ \widetilde{W}(x,r) := \frac{V_{\Z^m}(x, r)}{V_{\Z^k}(v_x,r)} \approx \frac{r^m}{r^k} = r^{m-k}.\]
Hence,
\[ \psi_K(x) \leq \sum_{n \geq d_x^2} \frac{C}{\widetilde{W}(x, \sqrt{n})} \approx \sum_{n\geq d^2} \frac{1}{n^{(m-k)/2}} \approx \frac{1}{d^{m-k-2}} \to 0 \quad \text { as } d \to \infty\]
since $(m-k)/2 > 1$ as $m-k \geq 3.$ Here there is only dependence on $d,$ the distance to $K=\Z^k,$ and not on $x$ itself. Hence $\Z^m$ is uniformly $S$-transient with respect to $\Z^k.$ In fact, in this case $\{ \sum_{n \geq 1} (W(x, \sqrt{n}))^{-1} \}_{x \in \widehat{\Gamma}\setminus K}$ converges uniformly. 

If instead $m-k < 3,$ then the series fails to converge, and Theorem \ref{psi-upperbd} gives the pointless bound of $\psi \leq \infty.$ 

\end{example}

\begin{remark}
In our examples, it is common that $K = \partial \Gamma.$ Theorem \ref{psi-upperbd} is useful for showing that $\widehat{\Gamma}$ is transient with respect to a subset $K$ of its vertices, and we have the idea that transient sets tend to be ``thin" or of smaller dimension, as we saw above, so the set $K$ doesn't have a much of an ``interior" in the larger graph. 

However, for future applications involving gluing graphs, thinking of the set $K$ as having some ``thickness" may be useful. Consider $\widehat{\Gamma} = \Z^m$ and take $K$ to be a cylinder, say $K = \{ \vec{x} \in \Z^m: |x_m| \leq r\},$ so that $K$ is the set of all points within distance $r$ from the $x_m$-axis. Then if $r \geq 1$ and $\Gamma = \Z^m \setminus K, \ \partial \Gamma \not = K.$ However, the chance we hit $K$ is essentially the same as the chance we hit a single line in $\Z^m,$ so $\Gamma$ is $S$-transient if and only if $m \geq 4.$ 
\end{remark}

\begin{example}[Sparse line in $\Z^3$]
In Example \ref{lattice_lattice}, we could not use Theorem \ref{psi-upperbd} to decide the $S$-transience/$S$-recurrence of a copy of $\Z$ in $\Z^3.$ The same can be said for a half-line in $\Z^3.$ Now consider the sparse half-line in $\Z^3$ given by dyadic points: $K = \partial \Gamma = \{(2^k, 0, 0) \in \Z^3: k \in \Z_{\geq 0} \}.$ Take the lazy SRW on $\Z^3.$ 

The hypotheses of Theorem \ref{psi-upperbd}/Corollary \ref{simpler} are satisfied; to verify doubling of $\partial \Gamma,$ note that $V_{\partial \Gamma}(z, r) \approx \log_2(r)$ for $z \in \partial \Gamma.$ Further note for any $v_x \in \partial \Gamma,$ we have $V_{\partial \Gamma}(v_x, r) \leq V_{\partial \Gamma}(0, r).$ We now compute, for any $x \in \Z^3,$ where $C$ can change from step to step,
\begin{align*}
\sum_{n \geq 1} \frac{1}{\widetilde{W}(x,\sqrt{n})} &= \sum_{n \geq 1} \frac{V_{\partial \Gamma}(v_x, \sqrt{n})}{V_{\Z^3}(x, \sqrt{n})} 
\leq \sum_{n \geq 1} \frac{V_{\partial \Gamma}(0, \sqrt{n})}{V_{\Z^3}(x, \sqrt{n})} \leq C \sum_{l \geq 0} \sum_{\sqrt{n} \asymp 2^l} \frac{V_{\partial \Gamma}(0,\sqrt{n})}{n^{3/2}} \\
&\leq C \sum_{l \geq 0} \sum_{\sqrt{n} \asymp 2^l} \frac{l+1}{2^{3l}} \leq C\sum_{l \geq 0} \frac{l+1}{2^{3l}} 2^{2l} =C  \sum_{l \geq 0} \frac{l+1}{2^l}.
\end{align*}
This is a convergent sum that is independent of $x.$ Therefore $\Z^3$ is uniformly $S$-transient with respect to the sparse line $K.$  
\end{example}

\begin{example}[Weighted half-spaces]\label{halfspace_trans} Consider $\widehat{\Gamma} = \Z^m_+ = \{ (x_1, \dots, x_m) \in \Z^m: x_m \geq 0 \}$ where $\pi(x_1, \dots, x_m) = (1+ x_m)^\alpha$ and $\alpha > 1.$ Let $K = \partial \Gamma = \{(x_1, \dots, x_m) \in \Z^m_+ : x_m = 0 \}.$ Let $\mathcal{K}_{\Z^m_+}$ denote the Markov kernel on $\Z^m_+$ where at each vertex away from the edge, the walk stays in place with probability $1/2$ or moves to a neighbor uniformly at random, and, at vertices on the boundary, the walk moves to a neighbor with probability $\frac{1}{2m}$ and otherwise stays in place (the probability of staying in place is $>1/2.$). Define a new Markov kernel on $\Z^m_+$ by 
\[ \mathcal{M}_{\Z^m_+}(x,y) = \begin{cases} \mathcal{K}_{\Z^m_+}(x,y) \min\{1, \frac{\pi(y)}{\pi(x)}\}, & x \not = y \\[1ex] 1 - \sum_{z \sim x} \mathcal{M}_{\Z^m_+}(x,z), & x=y.\end{cases} \]
Then this is a Markov kernel, and we consider the graph $(\Z^m_+, \pi, \mathcal{M}_{\Z^m_+}).$ Since $\mathcal{K}_{\Z^m}$ is symmetric with respect to the vertex measure that is identically $1,$ it is easy to verify $\mathcal{M}_{\Z^m_+}$ is symmetric with respect to $\pi.$ \\

The appropriate edge weights that give the same Markov kernel are
\[ \mu_{xy} = \pi(x) \mathcal{M}_{\Z^m_+}(x,y).\]
As $\mathcal{K}_{\Z^m_+}$ was uniformly lazy and had controlled weights, $\mathcal{M}_{\Z^m_+}$ inherits these properties. 

That $\Z^m_+$ is Harnack with respect to this random walk structure can be verified by directly showing it is doubling and satisfies the Poincar\'{e} inequality or by using arguments similar to those in Section 4.3 of \cite{ag_lsc_harnackstability}. Note the measure $\pi$ here is not bounded above. On $K,$ we have $\pi \equiv 1$ and $V_{\partial \Gamma}((x_1, \dots, x_{m-1}, 0),r) \approx r^{m-1}$ is doubling. Further, 
\begin{align*}
V_{\widehat{\Gamma}}((x_1, \dots, x_m),r) \approx \begin{cases} |x_m|^\alpha \, r^m, & r \leq |x_m| \\ r^{m+\alpha}, & r \geq |x_m|.\end{cases}
\end{align*}
The above computation can be seen as follows: If $r \leq |x_m|,$ there are approximately $r^m$ points in $B((x_1, \dots, x_m), r)$, each of which has weight approximately $|x_m|^\alpha$. It is clear we can get such an upper bound; for the lower bound, note the ball of radius $r$ contains the ball of radius $r/2$ which again has approximately $r^m$ points and since $r \leq |x_m|,$ all such points have weight approximately $|x_m|^\alpha.$ On the other hand, if $r \geq |x_m|,$ we can consider a ball of radius $r$ with center on $\partial \Gamma.$ As $\pi$ is constant except in the $x_m$ direction, the volume of such a ball is approximately $r^{m-1}(1^\alpha + \cdots + r^\alpha)$. As $\big(\frac{r}{2}\big)^\alpha \frac{r}{2} \leq 1^\alpha + \cdots + r^\alpha \leq r^\alpha \cdot r,$ the desired volume estimate follows. 

Therefore
\begin{align*}
\widetilde{W}((x_1, \dots, x_m), r) \approx \begin{cases} |x_m|^\alpha r, & r \leq |x_m| \\ r^{1+\alpha}, & r \geq |x_m|.\end{cases}
\end{align*}

Note that in this example, the family of sums $\{ \sum_{n\geq 1} (\widetilde{W}(x,\sqrt{n}))^{-1} \}_{x \in \Z_+^d \setminus K}$ does not converge uniformly as
\begin{align*}
\sum_{n \geq 1} \frac{1}{\widetilde{W}(x, \sqrt{n})} &\approx \sum_{n=1}^{|x_m|^2}\frac{1}{|x_m|^\alpha \, n^{1/2}} + \sum_{n> |x_m|^2} \frac{1}{n^{(1+\alpha)/2}} \\
&\leq \frac{1}{|x_m|^\alpha} \big[2|x_m| -2] + \Big(\frac{2}{\alpha-1}\Big) \frac{1}{|x_m|^{\alpha -1}} \quad \text{ (since } \alpha >1),
\end{align*}
which depends on $x.$ 

However, this example is still uniformly $S$-transient because $d_x = d(x,K) = |x_m|$ and the form of $\widetilde{W}$ changes based on comparing $r$ to $|x_m|$ in a convenient way so that 
\begin{align*}
\sum_{n \geq d_x^2} \frac{1}{\widetilde{W}(x, \sqrt{n})} \approx \sum_{n \geq |x_m|^2} \frac{1}{n^{(1+\alpha)/2}} \approx \frac{1}{d_x^{\alpha -1}} \to 0 \quad \text{ as } d_x \to \infty.
\end{align*}

Therefore for any $\varepsilon >0,$ we can choose $L > 0$ such that whenever $d(x,K) \geq L,$ we have 
\[ \psi_K(x) \leq \sum_{n > d_x^2} \frac{1}{\widetilde{W}(x, \sqrt{n})} = \frac{1}{d_x^{\alpha-1}} \leq \frac{1}{L^{\alpha-1}} < \varepsilon,\]
that is, the weighted half-space $\Z^m_+$ (with $\alpha >1$) is uniformly $S$-transient with respect to $K$ for all $m.$ 
\end{example}

\section{Harmonic Profiles and Hitting Probability Estimates}\label{h-trans}

The previous section obtained an upper bound for the hitting probability $\psi_K.$ In this section, we obtain two-sided bounds on $\psi_K.$ Getting a lower bound requires a better estimate on $\mathcal{K}_{\Gamma, D},$ which we will give in terms of a nice harmonic function (a harmonic profile) on $\Gamma.$ To guarantee such harmonic functions exist, we will make geometric assumptions about $\Gamma.$ 

Recall all graphs $(\Gamma, \pi, \mu)$ are assumed to be uniformly lazy and have controlled weights. 

\begin{definition}[Uniform]
A subgraph $\Gamma$ of a graph $\widehat{\Gamma}$ is \emph{uniform} in $\widehat{\Gamma}$  if there exist constants $0<c_u, \, C_U < +\infty$ such that for any $x,y \in \Gamma$ there is a path $\gamma_{xy} = (x_0 = x, x_1, \dots, x_k =y)$ between $x$ and $y$ in $\Gamma$ such that 
\begin{enumerate}[(a)]
\item $k \leq C_U d_{\widehat{\Gamma}}(x,y) $
\item For any $j \in \{0, \dots, k\},$
\[ d_{\widehat{\Gamma}}(x_j, \partial \Gamma) = d_{\widehat{\Gamma}}(x, \widehat{\Gamma} \setminus \Gamma) \geq c_u (1 + \min \{ j, k-j\}).\]
\end{enumerate}
\end{definition}

\begin{definition}[Inner Uniform]
A subgraph $\Gamma$ of $\widehat{\Gamma}$ is \emph{inner uniform} in $\widehat{\Gamma}$  if there exist constants $0<c_u, \, C_U < +\infty$ such that for any $x,y \in \Gamma$ there is a path $\gamma_{xy} = (x_0 = x, x_1, \dots, x_k =y)$ between $x$ and $y$ in $\Gamma$ such that 
\begin{enumerate}[(a)]
\item $k \leq C_U d_{\Gamma}(x,y) $
\item For any $j \in \{0, \dots, k\},$
\[ d_{\widehat{\Gamma}}(x_j, \partial \Gamma) = d_{\widehat{\Gamma}}(x, \widehat{\Gamma} \setminus \Gamma) \geq c_u (1 + \min \{ j, k-j\}).\]
\end{enumerate}
\end{definition}

The only difference between a uniform domain and an inner uniform domain is that uniform domains require the length of the path in $\Gamma$ to be comparable to the distance between $x$ and $y$ in the larger graph $\widehat{\Gamma},$ while an inner uniform domain requires the length of the path to be comparable to distance in $\Gamma.$ This somewhat subtle difference is key. Recall $d_\Gamma(x_j, \partial \Gamma) = d_{\widehat{\Gamma}}(x_j, \partial \Gamma)$ if we extend $d_\Gamma$ to $\partial \Gamma.$ We refer the reader to \cite{pd_khe_lsc_finiteMC} and the references therein for more details on such geometric assumptions on domains in the discrete space setting; in particular, Section 8.1 gives many examples of finite (inner) uniform domains. Condition (b) in these definitions can be thought of as a ``banana'' or ``cigar'' condition and says that it must be possible to fit a linearly growing ``banana'' (with respect to the distance to the end points) around all paths from $x$ to $y$. This ``banana'' must stay inside the domain. 

All uniform domains are inner uniform. Domains above Lipschitz functions in $\Z^d$ are uniform. A slit two-dimensional lattice is the typical example of a domain that is inner uniform but not uniform. Similarly, the complement of a discrete parabola in $\Z^2$ is inner uniform but not uniform. In general, slits and ``bottlenecks'' are obstacles to uniformity. An example of a domain that is neither inner uniform nor uniform is $\{(x,y) \in \Z^2: x \leq -2\} \cup \{(x,y) \in \Z^2: x \geq 2\} \cup \{ (-1,0), (0,0), (1,0)\},$ considered as a subgraph of $\Z^2.$ There is no criterion to determine (inner) uniformity, and proving whether a given set is (inner) uniform should be thought of as difficult. 

Inner uniform domains are useful because they allow us to transfer the Harnack inequality from a larger graph to a subgraph. 
\begin{theorem}[Theorem 1.10 of \cite{Kelsey_thesis}]\label{IU_Harnack}
Let $(\widehat{\Gamma}, \mathcal{K}, \pi)$ be a Harnack graph and $\Gamma$ be an inner uniform subgraph of $\widehat{\Gamma}.$ Then $(\Gamma, \mathcal{K}_{\Gamma,N}, \pi)$ is also a Harnack graph. 
\end{theorem}

\begin{remark}
The converse of Theorem \ref{IU_Harnack} is not true. For instance, consider the traces of two parabolas in $\Z^2$ (with the lazy simple random walk) connected by a finite number of edges. One such example is $\Gamma = \{ (x,y) \in \Z^2: y \geq x^2 +1\} \cup  \{ (x,y) \in \Z^2: y \leq -x^2 -1\} \cup \{(0,0)\},$ where this denotes the vertex set of a subgraph of $\Z^2$. The continuous version of this example is Harnack by Theorem 7.1 of \cite{ag_lsc_harnackstability}. Therefore, the discrete version is also Harnack by results of \cite{TC_LSC_IsoInfini}. This is an example of a subgraph of a Harnack graph where the subgraph is neither uniform nor inner uniform but $(\Gamma, \mathcal{K}_{\Gamma,N}, \pi)$ is nonetheless Harnack.
\end{remark}

\begin{definition}[Harmonic Profile]
A function $h$ is an \emph{harmonic profile} for an infinite graph $\Gamma$ that is a subgraph of $\widehat{\Gamma}$ if it satisfies the following properties:
\begin{enumerate}[1.]
\item $h > 0$ in $\Gamma$
\item $h = 0$ on the exterior boundary of $\Gamma$
\item $h$ is harmonic in $\Gamma,$ that is, 
\begin{align}\label{harmonic_graph} h(x) &= \sum_{y \in \Gamma} \mathcal{K}_{\widehat{\Gamma}}(x,y) h(y) = \sum_{ y \in \Gamma} \mathcal{K}_{\Gamma,D}(x,y) h(y) \quad \forall x \in \Gamma.\end{align}

(Note $\mathcal{K}_{\Gamma,D}(x,y) = 0$ unless $y \simeq x$ and $h(y)=0$ if $y \not \in \Gamma.$) 

\end{enumerate}
\end{definition}

On finite graphs, there is no such function satisfying properties 1., 2., and 3. above since any harmonic function that is zero on the exterior boundary of $\Gamma$ (which we assume to be non-empty) is zero everywhere. 

We would like to appeal to a variety of pre-existing results about the existence of harmonic profiles and their properties in inner uniform domains. In the continuous space setting, the desired results are found in \cite{bluebook}. These results were transferred to the graph setting in the case of infinite graphs in \cite[Chapter 5]{Kelsey_thesis}; see also \cite[Chapter 8]{pd_khe_lsc_finiteMC}. In general, the technique of \cite{Kelsey_thesis} is to associate with any given graph its cable process. The cable process takes place in a continuous space with a nice Dirichlet form, so the results of \cite{bluebook} apply to it, and there is a one-to-one correspondence between a profile of the cable process and a profile of the graph. 

\begin{proposition}[Prop. 5.1, Corollary 5.3 of \cite{Kelsey_thesis}]
Suppose $\Gamma$ is a proper infinite subgraph of $(\widehat{\Gamma}, \pi, \mu).$ Then there exists a harmonic profile $h$ for $\Gamma.$ Moreover, if $\Gamma$ is inner uniform in $\widehat{\Gamma}$ and $\widehat{\Gamma}$ is Harnack, then the profile $h$ of $\Gamma$ is unique up to multiplication by a constant. 
\end{proposition}

The existence of $h$ is straightforward. The uniqueness of $h$ is more subtle and can be obtained from \cite[Theorem 4.1]{bluebook} via the cable process.

\subsection{\texorpdfstring{$h$}{h}-transform on graphs}\label{h_transform}

The existence of the profile $h$ for a graph $\Gamma$ (considered as a subgraph of a graph $\widehat{\Gamma}$) allows us to consider a reweighted version of $\Gamma,$ which we will refer to as the $h$-transform space. Recall a graph and the random walk structure on it may be given by triples of the form $(\Gamma, \pi, \mu)$ or $(\Gamma, \mathcal{K}, \pi).$ 

Reweight the measure $\pi$ on $\Gamma$ by $h^2$ to obtain the measure $\pi_h(x) = h^2(x) \pi(x).$ As $h>0$ on $\Gamma, \ \pi_h$ remains a positive function on the vertices of $\Gamma.$ Define a Markov kernel by
\begin{align}\label{K_h} 
\mathcal{K}_h(x,y) = \frac{1}{h(x)} \mathcal{K}_{\Gamma,D}(x,y) h(y). 
\end{align}
That this is a Markov kernel follows since $h$ is harmonic as, for all $x \in \Gamma,$
\begin{align*}
\sum_{y \in \Gamma} \mathcal{K}_h(x,y) &= \sum_{y \in \Gamma} \frac{1}{h(x)} \mathcal{K}_{\Gamma,D}(x,y) h(y)=\frac{h(x)}{h(x)}=1.
\end{align*}

Notice $\mathcal{K}_h$ is a Markov kernel, despite that not being the case for $\mathcal{K}_{\Gamma,D}.$ Thus one effect of the $h$-transform is to return us to the setting of Markov kernels as opposed to subMarkovian ones. 

Moreover, $\mathcal{K}_h$ is reversible with respect to $\pi_h$ since $\mathcal{K}_{\Gamma,D}$ is reversible with respect to $\pi$:
\begin{align*}
\mathcal{K}_h(x,y) \pi_h(x) &= \mathcal{K}_h(x,y) h(x)^2 \pi(x)  = h(x) h(y) \mathcal{K}_{\Gamma,D}(x,y)\pi(x) \\&= h(x) h(y) \mathcal{K}_{\Gamma,D}(y,x) \pi(y) = \mathcal{K}_h(y,x) h^2(y)\pi(y) \\&= \mathcal{K}_h(y,x) \pi_h(y).
\end{align*}

Directly giving a formula for $\mathcal{K}_h$ as in (\ref{K_h}) is equivalent to taking reweighted conductances $\mu^h_{xy} = h(x)h(y) \mu_{xy}$ on $\Gamma$ and then defining the Markov kernel as in Section \ref{subgraph}. Note that $\mu^h_{xy} = 0$ if at least one of $x,y \not \in \Gamma,$ so it does not matter whether we think of the Neumann or Dirichlet kernel. Considering the graph this way, the weights $\mu_{xy}^h$ are subordinate to the measure $\pi_h$ due to the harmonicity of $h.$ Further, if $\widehat{\Gamma}$ has controlled weights, the same holds for the $h$-transform space since $h(y)/h(x)$ is bounded below for $x \sim y$ ($x, y \in \Gamma$). Note $\mathcal{K}_h(x,x) = \mathcal{K}_{\Gamma, D}(x,x) = \mathcal{K}_{\widehat{\Gamma}}(x,x),$ so the $h$-transform graph is uniformly lazy if and only if that is true of $\widehat{\Gamma}.$

The heat kernel $p_h(n,x,y)$ on the $h$-transform of $\Gamma$ is the transition density of $\mathcal{K}_h$ and is given by $\mathcal{K}^n_h(x,y)/\pi_h(y).$ The $h$-transform heat kernel on $\Gamma$ and the Dirichlet heat kernel on $\Gamma$ have the following relationship: 
\begin{align*}
p_h(n,x,y) = \frac{\mathcal{K}^n_h(x,y)}{\pi_h(y)} = \frac{\mathcal{K}^n_{\Gamma,D}(x,y)}{h(x) h(y) \pi(y)} = \frac{1}{h(x)h(y)} p_{\Gamma,D}(n,x,y).
\end{align*} 

Under certain conditions, we have good two-sided estimates for the heat kernel of the $h$-transform of $\Gamma,$ which is the content of the next theorem.

\begin{theorem}[Theorem 1.11 and Corollary 5.8 of \cite{Kelsey_thesis}]\label{ph_Harnack}
Suppose $(\widehat{\Gamma},\pi,\mu)$ is a Harnack graph and $\Gamma$ is an inner uniform subgraph of $\widehat{\Gamma}.$ Then $(\Gamma, \mathcal{K}_h, \pi_h)$ is also a Harnack graph. Consequently, there exist constants $c_1,c_2,c_3,c_4>0$ such that, for all $x,y \in \Gamma$ and $n \geq d_\Gamma(x,y),$ 
\begin{align*}
\frac{c_1}{V_h(x, \sqrt{n})} \exp\Big(-\frac{ d_\Gamma^2(x,y)}{c_2 n}\Big) \leq p_h(n,x,y) \leq \frac{c_3}{V_h(x, \sqrt{n})} \exp\Big(-\frac{ d_\Gamma^2(x,y)}{c_4 n}\Big)
\end{align*}
or, equivalently, 
\begin{align*}
\frac{c_1h(x)h(y)}{V_h(x, \sqrt{n})} \exp\Big(-\frac{ d_\Gamma^2(x,y)}{c_2 n}\Big) \leq p_{\Gamma,D}(n,x,y) \leq \frac{c_3 h(x)h(y)}{V_h(x, \sqrt{n})} \exp\Big(-\frac{ d_\Gamma^2(x,y)}{c_4 n}\Big)
\end{align*} 
Here $V_h$ denotes the volume in $\Gamma$ with respect to the measure $\pi_h.$
\end{theorem}

The following lemma is useful for computing $V_h$.

\begin{lemma}[{\cite[Proposition 5.5]{Kelsey_thesis}}]\label{compute_Vh}
Let $\Gamma$ be inner uniform in a Harnack graph $(\widehat{\Gamma},\pi,\mu).$ For any $x \in \Gamma, \ r>0,$ let $x_r \in \Gamma$ be a point such that $d_\Gamma(x,x_r) \leq r/4$ and $d(x_r, \partial \Gamma) \geq c_u r/8$ (recall $c_u$ is one of the inner uniformity constants). Then there exist constants $c,C$ (independent of $x,r$) such that 
\begin{align*}
c h(x_r)^2 V_\Gamma(x,r) \leq V_h(x,r) \leq C h(x_r)^2 V_\Gamma(x,r). 
\end{align*}
\end{lemma}

\begin{remark}
The existence of points $x_r$ as in the lemma above is a relatively straightforward consequence of the inner uniform assumption (see \cite[Lemma 3.20]{bluebook}, \cite[Lemma 4.7]{Kelsey_thesis}). 
\end{remark}

\begin{remark}\label{Carleson_est}
The definition of such points $x_r$ is motivated by the following key property of $h$: There exists a constant $A$ such that
\[ h(y) \leq A h(x_r) \quad \forall \, r>0,\ y \in B_\Gamma(x,r).\]
This property is called a Carleson estimate, and it follows from arguments given in Section 4.3.3 (in particular (4.28)) of \cite{bluebook}, Chapter 8 of \cite{pd_khe_lsc_finiteMC}, and Theorem 2 of \cite{aikawa}. This property is crucial to Lemma \ref{compute_Vh}. 

Moreover, due to the harmonicity of $h$ and the inner uniform property, $h(x_{2r}) \approx h(x_r),$ and $V_\Gamma$ is doubling. 
\end{remark}

\begin{remark}
In the situation where we can compute $h$, the above abstract examples become concrete. For example, if $\widehat{\Gamma} = \Z^m$ and $\Gamma = \{ (x_1, \dots, x_m) \in \Z^m: x_m> 0\}$ is the upper half-space, then $h(x_1,\dots,x_m) = x_m.$ It is easy to verify the above claims about $h$ for this example. However, there are only a few situations where exact formulas for $h$ are known, and, in general, estimating $h$ is a hard problem. 
\end{remark}

The following theorem holds for continuous spaces and is discussed in Chapter 4 of \cite{bluebook}. Once again, the theorem can be transferred to the discrete setting using the cable process (see \cite{Kelsey_thesis}).

\begin{theorem}[Boundary Harnack Principle {\cite[Theorem 4.18]{bluebook}}]\label{bdry_Harnack}
Assume $\Gamma$ is an inner uniform subgraph of the Harnack graph $(\widehat{\Gamma}, \pi, \mu).$ Then there exist constants $A_0, A_1 \in (1,\infty)$ such that for any $\xi \in \partial_I \Gamma$ and any two positive harmonic functions $f,g$ on $B_{\Gamma}(\xi, A_0 r)$ that are zero along $\partial \Gamma \cap B_{\widehat{\Gamma}}(\xi, A_0 r),$ we have 
\begin{align*}
\frac{f(x)}{f(x')} \leq A_1 \frac{g(x)}{g(x')} \quad \forall x, x' \in B_{\Gamma}(\xi, r).
\end{align*}
\end{theorem}

\subsection{Hitting Probabilities and Dirichlet Kernels in the Inner Uniform Case} 

Theorem \ref{ph_Harnack} gave two-sided estimates of $p_{\Gamma,D}$ in terms of $h;$ whenever we can estimate $h$ on part (or all) of $\Gamma,$ the abstract estimate of $p_{\Gamma,D}$ becomes more concrete.

\begin{lemma}[Behavior of $h$ in transient case]\label{his1}
Let $\Gamma := \widehat{\Gamma} \setminus K$ be inner uniform in the Harnack graph $(\widehat{\Gamma}, \pi, \mu).$  If $\widehat{\Gamma}$ is $S$-transient with respect to $K,$  then the profile $h$ of $\Gamma$ is given by $1- \psi_K.$ If $\widehat{\Gamma}$ is uniformly $S$-transient with respect to $K,$ then $h \approx 1$.
\end{lemma}

\begin{proof}
Since $\psi_K$ is the hitting probability of $K,$ and the exterior boundary of $\Gamma$ is contained in $K, \ \psi_K$ is harmonic in $\Gamma.$ Further, $0\leq \psi_K \leq 1$ on $\Gamma$ and $\psi_K \equiv 1$ on $K.$ Hence $h = 1 - \psi_K$ is a harmonic function inside of $\Gamma$ that is zero on the exterior boundary of $\Gamma$. Since $\widehat{\Gamma}$ is $S$-transient with respect to $K,$ there exists some $y \in \Gamma$ such that $\psi_K(y) <1.$ Thus $h(y) > 0,$ and, by the maximum principle, $h(x) >0$ for all $x \in \Gamma.$ Therefore $h$ is the profile of $\Gamma.$

Now suppose $\widehat{\Gamma}$ is uniformly $S$-transient with respect to $K.$ Then there exist $L, \varepsilon >0$ such that $\psi_K(x) \leq 1- \varepsilon$ whenever $d(x,K) \geq L.$ (Note the distance from $x$ to $K$ is the same whether considered in all of $\widehat{\Gamma}$ or only in $\Gamma$.) Hence for $d(x,K) \geq L,$ we have $\varepsilon \leq 1- \psi_K(x) = h(x).$ From the definition of a harmonic function, $h(x) \geq (1/C_c) h(y)$ for $x\sim y, \ x,y \in \Gamma,$ where $C_c$ is the constant for controlled weights. Applying this inequality a finite number of times (since $d(x,K) \geq 1 \ \forall x \in \Gamma),$ there exists $\varepsilon_* >0$ such that $\varepsilon_* \leq h(x) \leq 1$ for all $x \in \Gamma.$
\end{proof}

\begin{corollary}\label{D_N_similar}
Assume that $\widehat{\Gamma}$ is a Harnack graph that is uniformly  $S$-transient with respect to $K$ and that $\Gamma := \widehat{\Gamma} \setminus K$ is inner uniform. Then there exist constants $0<c, \, C < +\infty$ such that 
\[ cp_{\Gamma,N}(Cn, x, y) \leq p_{\Gamma,D}(n,x,y) \leq p_{\Gamma,N}(n,x,y). \]

If $\widehat{\Gamma}$ is $S$-transient with respect to $K$, then the Neumann and Dirichlet heat kernels are comparable in the region where $h \approx 1.$ 
\end{corollary}

In other words, adding killing along $\Gamma$ does not significantly alter the behavior of the heat kernel in this setting. The corollary above can be compared with Theorem 3.1 of \cite{ag_lsc_extcptset}, where a similar result is obtained for Riemannian manifolds when $K$ is compact and the manifold is transient.

\begin{proof}
The upper bound is immediate. Since we are in the setting where $\Gamma$ is an inner uniform subgraph of a Harnack $\widehat{\Gamma},$ by Theorem \ref{IU_Harnack}, $(\Gamma, \mathcal{K}_{\Gamma,N}, \pi)$ is a Harnack graph. Thus there exist constants $c_1, c_2, c_3, c_4 >0$ such that for all $x,y\in \Gamma$ and all $n \geq d(x,y),$ 
\begin{align*}
\frac{c_1}{V(x,\sqrt{n})} \exp\Big(-\frac{d_\Gamma^2(x,y)}{c_2 n}\Big) \leq p_{\Gamma,N}(n,x,y) \leq \frac{c_3}{V(x,\sqrt{n})} \exp\Big(-\frac{d_\Gamma^2(x,y)}{c_4 n}\Big).
\end{align*} 
From Theorem \ref{ph_Harnack}, we also know that the $h$-transform of $\Gamma$ is Harnack. In the uniformly $S$-transient setting,  $h \approx 1$ by Lemma \ref{his1}. Therefore $h(x) \approx h(y) \approx 1$ and $V_h \approx V.$ Hence there exist constants $b_1, b_2, b_3, b_4>0$ such that 
\begin{align*}
\frac{b_1}{V(x,\sqrt{n})} \exp\Big(-\frac{d_\Gamma^2(x,y)}{b_2 n}\Big) \leq p_{\Gamma,D}(n,x,y) \leq \frac{b_3}{V(x,\sqrt{n})} \exp\Big(-\frac{d_\Gamma^2(x,y)}{b_4 n}\Big).
\end{align*} 
Hence $p_{\Gamma,N}, \ p_{\Gamma, D}$ satisfy two-sided Gaussian estimates and we obtain the desired lower bound. This argument holds whenever $h(x), h(y) \approx 1,$ so the statement about the transient case follows.
\end{proof}

\begin{theorem}[Two-sided estimates on $\psi_K$]\label{IU-psibd}
Suppose that $\Gamma:= \widehat{\Gamma} \setminus K$ is inner uniform in the Harnack graph $(\widehat{\Gamma}, \pi, \mu).$ 

Then, where the constants for $\approx$ depend on the constants appearing in the inner uniform, Harnack, controlled weights, and uniformly lazy assumptions, 
\begin{align}
\psi_K(x) \approx \sum_{n \geq d_{\Gamma}^2(x,\partial_I \Gamma)} \sum_{\substack{y \in \partial_I \Gamma: \\ d_{\Gamma}^2(x,y) \leq n}} \frac{h(x)}{h(x_{\sqrt{n}})} \frac{h(y)}{h(y_{\sqrt{n}})} \frac{\pi(y)}{V_\Gamma(x, \sqrt{n})} \quad \forall x \in \Gamma \setminus \partial_I \Gamma.
\end{align}

If, in addition, $\Gamma$ is uniformly $S$-transient, then the two-sided bound 
\begin{align}\label{2sided_trans}
\psi_K(x) \approx \sum_{n \geq d_\Gamma^2(x, \partial_I \Gamma)} \sum_{\substack{y \in \partial_I \Gamma: \\ d_{\Gamma}^2(x,y) \leq n}}  \frac{\pi(y)}{V_\Gamma(x, \sqrt{n})}
\end{align}
holds, where the constants in $\approx$ are as above and also depend upon $L, \varepsilon$ from the uniformly $S$-transient assumption.
\end{theorem} 

\begin{remark}
In Theorem \ref{psi-upperbd}, the main step of the proof that resulted in an upper bound (without a matching lower bound) came from using the inequality $\mathcal{K}_{\Gamma,D}^{n-1} \leq \mathcal{K}_{\Gamma,N}^{n-1}.$ In Theorem \ref{psi-upperbd}, no assumptions about the geometry of $\Gamma$ (or $K$) were made, and the proof uses $d_{\widehat{\Gamma}}.$ Theorem \ref{IU-psibd} instead uses the distance $d_\Gamma,$ so while these theorems are similar, the main objects differ. If $\Gamma$ is \emph{uniform} in $\widehat{\Gamma},$ then $d_\Gamma \approx d_{\widehat{\Gamma}}.$ However, even under uniformity, the upper bounds of Theorem \ref{psi-upperbd} and Theorem \ref{IU-psibd} only clearly agree up to a constant if there is also some sort of doubling of the set $\partial \Gamma$ as in Corollary \ref{simpler}.
\end{remark}

\begin{proof}
For $x \in \Gamma \setminus \partial_I \Gamma, \  d_\Gamma(x, K) \geq 2$ and
\begin{align*}
\psi_K(x) &= \bP^x(\tau_K < + \infty) = \sum_{n\geq 2} \sum_{v \in \partial \Gamma} \sum_{\substack{y \sim v \\ y \in \Gamma}} \mathcal{K}_{\Gamma,D}^{n-1}(x,y) \mathcal{K}_{\widehat{\Gamma}}(y,v)\\ 
&= \sum_{n\geq 2} \sum_{v \in \partial \Gamma} \sum_{\substack{y \sim v \\ y \in \Gamma}} \frac{h(x)}{h(y)} \mathcal{K}_{h}^{n-1}(x,y) \mathcal{K}_{\widehat{\Gamma}}(y,v)  \\
&\approx  \sum_{n\geq 2} \sum_{v \in \partial \Gamma} \sum_{y \sim v} \frac{ h(x) h(y) \pi(y)}{V_h(x, \sqrt{n})} \exp\Big( \frac{- d_\Gamma^2(x,y)}{n}\Big) \\
&\approx \sum_{n\geq 2} \sum_{y \in \partial_I \Gamma} \frac{ h(x) h(y) \pi(y)}{V_h(x, \sqrt{n})} \exp\Big( \frac{- d_\Gamma^2(x,y)}{n}\Big),
\end{align*}
where we have used that $\mathcal{K}_{\widehat{\Gamma}}(y,v)$ is roughly constant (by the controlled weight hypothesis), the result of Theorem \ref{ph_Harnack} for $\mathcal{K}_h^{n-1},$ and that each $y \in \partial_I \Gamma$ is adjacent to at least one, but at most finitely many, $v \in \partial \Gamma$ (uniformly over $y$). 

Since $(\Gamma, \mathcal{K}_h, \pi_h)$ is a Harnack graph, it must be doubling and satisfy the Poincar\'{e} inequality. Taking the sum in time $n$ and using Lemma \ref{compute_Vh} to estimate $V_h,$ 
\begin{align*}
\sum_{n \geq 2} \frac{1}{V_h(x, \sqrt{n})} \exp\Big(-\frac{ d_\Gamma^2(x,y)}{n}\Big) \approx \sum_{n \geq d_\Gamma^2(x,y)} \frac{1}{V_h(x, \sqrt{n})} \approx \sum_{n \geq d_\Gamma^2(x,y)} \frac{1}{h(x_{\sqrt{n}})^2 V_\Gamma(x, \sqrt{n})} ,
\end{align*}
where the upper bound follows from the same argument as in Theorem \ref{psi-upperbd} and the lower bound comes from forgetting the earlier terms of the sum. 

If $d_{\Gamma}(x,y) \leq \sqrt{n},$ then $h(x_{\sqrt{n}}) \approx h(y_{\sqrt{n}})$. This follows from the inequality
\begin{align*}
 h(y_{\sqrt{n}})^2 V_\Gamma(y, \sqrt{n}) &\leq C V_h(y, \sqrt{n}) \leq C V_h(x, 2 \sqrt{n}) \leq  C h(x_{\sqrt{n}})^2 V_\Gamma(x, 2 \sqrt{n}) \\&\leq C h(x_{\sqrt{n}})^2 V_\Gamma(y, \sqrt{n}),
\end{align*}\
where we have used the relationship between $V_h$ and $V_\Gamma$ and that both of these are doubling. 

Hence 
\begin{align*}
\psi_K(x) \approx \sum_{y \in \partial_I \Gamma} \sum_{n \geq d_\Gamma^2(x,y)} \frac{h(x)}{h(x_{\sqrt{n}})} \frac{h(y)}{h(y_{\sqrt{n}})} \frac{\pi(y)}{V_\Gamma(x, \sqrt{n})}.
\end{align*}

Now interchange the order of summation. Noting the set $\{y \in \partial_I \Gamma: d_\Gamma^2(x,y) \leq n \}$ is nonempty if and only if $n \geq d_\Gamma^2(x,y_x),$
\[ \sum_{y \in \partial_I \Gamma} \sum_{n \geq d_\Gamma^2(x,y)} \longleftrightarrow \sum_{n \geq d^2_\Gamma(x,y_x)} \sum_{\substack{y \in \partial_I \Gamma:  \\ d_\Gamma^2(x,y) \leq n}} .\]

Thus
\begin{align*} 
\psi_K(x) \approx \sum_{n \geq d^2_\Gamma(x,y_x)} \sum_{\substack{y \in \partial_I \Gamma:  \\ d_\Gamma^2(x,y) \leq n}}\frac{h(x)}{h(x_{\sqrt{n}})} \frac{h(y)}{h(y_{\sqrt{n}})} \frac{\pi(y)}{V_\Gamma(x, \sqrt{n})}
\end{align*} 

The result for the uniformly $S$-transient case follows from the above and Lemma \ref{his1}. 
\end{proof} 

\begin{remark}
Recall from Remark \ref{Carleson_est} that the Carelson estimate $h(z) \leq A h(x_r)$ holds for all $r>0, \ z \in B_\Gamma(x,r)$. Therefore the terms $h(x)/h(x_{\sqrt{n}})$ and $h(y)/h(y_{\sqrt{n}})$ are bounded and, essentially,  add additional decay to the sum. Thus Theorem \ref{IU-psibd} has additional decay that is not present in Theorem \ref{psi-upperbd}. However, if $h \approx 1, \ d_\Gamma \approx d_{\widehat{\Gamma}},$ and the boundary is doubling, then these bounds are the same. 

In the $S$-recurrent case, $\psi_K \equiv 1,$ so the two-sided bound above yields constants; see Example \ref{Z3_line} below.
\end{remark} 

\begin{theorem}\label{almost_converse}
Suppose that $\Gamma := \widehat{\Gamma} \setminus K$ is an inner uniform subgraph of the Harnack graph $(\widehat{\Gamma},\pi,\mu)$.

Let $B_{\partial_I}(x,r) := B_\Gamma(x, r) \cap \partial_I \Gamma$ denote the trace of $\Gamma$-balls in $\partial_I \Gamma$ and $V_{\partial_I}(x,r) = \pi(B_{\partial_I}(x,r))$ for any $x \in \Gamma.$ Define 
\[ W_{\partial_I}(x,r) := \frac{V_\Gamma(x,r)}{V_{\partial_I}(x,r)} \quad \forall x \in \Gamma.\]

 Then:
\begin{enumerate}[(1)]
\item If $\Gamma$ is uniformly $S$-transient, there exists some $L', \, \varepsilon' >0$ such that 
\[ d(x, K) \geq L' \implies  \sum_{n \geq d_\Gamma^2(x, \partial_I \Gamma)} \frac{1}{W_{\partial_I}(x,\sqrt{n})} \leq \varepsilon '.\]

\item If for any $\varepsilon >0,$ there exists $L_\varepsilon >0$ such that 
\[ d(x, K) \geq L_\varepsilon \implies \sum_{n \geq d_\Gamma^2(x, \partial_I \Gamma)} \frac{1}{W_{\partial_I}(x, \sqrt{n})} \leq \varepsilon,\]
then $\Gamma$ is uniformly $S$-transient. 

\end{enumerate}
\end{theorem} 

\begin{proof}
($1$):  Suppose that $\Gamma$ is uniformly $S$-transient, so there exist $\varepsilon, L >0$ such that $\psi_K(x) \leq 1-\varepsilon$ whenever $d(x, K) \geq L.$ By Lemma \ref{his1}, we have $h \approx 1.$ Using the result of Theorem \ref{IU-psibd}, 
\begin{align*}
\sum_{n \geq d_\Gamma^2(x, \partial_I \Gamma)} \frac{1}{W_{\partial_I}(x, \sqrt{n})} &= \sum_{n \geq d_{\Gamma}^2(x,\partial_I \Gamma)} \sum_{\substack{y \in \partial_I \Gamma: \\ d_{\Gamma}^2(x,y) \leq n}} \frac{\pi(y)}{V_\Gamma(x, \sqrt{n})} 
\leq C \psi_K(x) \leq C(1- \varepsilon)
\end{align*}
whenever $d(x,K) \geq L$. Setting $\varepsilon' = C(1-\varepsilon)$ gives the result.

($2$): Set $d_{x,I} := d(x, \partial_I \Gamma).$ Now suppose that for any $\varepsilon >0,$ there exists $L_\varepsilon >0$ such that $\sum_{n \geq d_{x,I}^2} (W_{\partial_I}(x, \sqrt{n}))^{-1} < \varepsilon$ whenever $d(x,K) \geq L_\varepsilon.$ Using Theorem \ref{IU-psibd} and the fact that $h(x)/h(x_{\sqrt{n}}) \leq 1$ for all $x \in \Gamma, \sqrt{n} \geq 1,$
\begin{align*}
\psi_K(x) &\leq C \sum_{n \geq d_{\Gamma}^2(x,\partial_I \Gamma)} \sum_{\substack{y \in \partial_I \Gamma: \\ d_{\Gamma}^2(x,y) \leq n}} \frac{h(x)}{h(x_{\sqrt{n}})} \frac{h(y)}{h(y_{\sqrt{n}})} \frac{\pi(y)}{V_\Gamma(x, \sqrt{n})}  \\
&\leq C \sum_{n \geq d_{\Gamma}^2(x,\partial_I \Gamma)} \sum_{\substack{y \in \partial_I \Gamma: \\ d_{\Gamma}^2(x,y) \leq n}} \frac{\pi(y)}{V_\Gamma(x, \sqrt{n})} \\
&= \sum_{n \geq d_\Gamma^2(x, \partial_I \Gamma)} \frac{C}{W_{\partial_I}(x, \sqrt{n})} \leq C \varepsilon.
\end{align*}
Taking $\varepsilon$ sufficiently small, there exists $\varepsilon', \, L'>0$ such that $\psi_K(x) \leq 1 - \varepsilon'$ whenever $d(x,K) \geq L',$ which is precisely the definition of $\Gamma$ being uniformly $S$-transient. 
\end{proof}

\begin{remark}
Theorem \ref{almost_converse} relies upon the lower bound of Theorem \ref{IU-psibd} in (1) and the upper bound in (2). An analogous statement of (2) could be obtained in the setting of Theorem \ref{psi-upperbd} using the function $W$ as opposed to the function $W_{\partial_I}.$ 
\end{remark}

\subsection{Two-sided bounds on hitting probabilities accounting for time or vertex hit} 

When $\Gamma$ is an inner uniform subgraph of a Harnack graph $\widehat{\Gamma},$ Theorem \ref{IU-psibd} gives matching upper and lower bounds on the probability of leaving $\Gamma$ (i.e. the probability of hitting $\Gamma^c$).  Other questions of natural interest include the likelihood of exiting $\Gamma$ at a particular point $v \in \partial \Gamma,$ or the chance of exiting $\Gamma$ (in general, or at a particular point) at or before time $n.$ While in the recurrent case $\psi_K(x) \equiv 1,$ this is not the case for the probabilities in the previous sentence, and these questions remain interesting. Bounds on these probabilities can be given using the same ideas and reasoning we have already seen. We collect these results below as corollaries. In particular, Corollaries \ref{hitprob_deriv} and \ref{hitv_beforen} can be seen as discrete versions of the results of \cite{ag_lsc_hittingprob}. 

\begin{definition}[Various Hitting Probabilities]
Given a graph $(\widehat{\Gamma}, \pi, \mu)$ with subgraph $\Gamma = \widehat{\Gamma} \setminus K,$ recall $\tau_K$ denotes the first hitting time of $K$/first exist time of $\Gamma$. Define the following hitting probabilities, where $x \in \Gamma, v \in \partial \Gamma,$ and $n\geq d_\Gamma(x,v):$ 
\begin{itemize}
\item $\psi_K(x) = \bP^x(\tau_K  < + \infty),$ the chance of hitting $K,$ given the walk starts at $x$ 
\item $\psi_K(x,v) = \bP^x (X_{\tau_K} = v, \, \tau_K < + \infty),$ the chance of hitting $K$ for the first time at the point $v,$ given the walk starts at $x$
\item $\psi_K(n,x,v) = \bP^x (X_{\tau_K} =v, \, \tau_K \leq n),$ the chance of hitting $K$ for the first time at the point $v \in \partial \Gamma,$ and doing so in time less than or equal to $n$ 
\item $\psi_K'(m,x,v) = \psi_K(m, x,v) - \psi_K(m-1,x,v) = \bP^x (X_{\tau_K} = v,\, \tau_K =m),$ the chance of hitting $K$ for the first time at $v$ at the time $m$
\item $\psi_K(n,x) = \bP^x (\tau_K \leq n)$ the chance of hitting $K$ at time less than or equal to $n$
\item $\psi_K'(m,x) = \psi_K(m,x) - \psi_K(m-1, x) = \bP^x(\tau_K = m),$ the chance of hitting $K$ for the first time at time $m$.
\end{itemize}
\end{definition}

There are various relationships between these quantities, for example
\begin{align*}
&\psi_K(n,x,v) = \sum_{m=0}^n \psi_K'(m,x,v) \\
&\psi_K(n,x) = \sum_{m=0}^n \psi_K'(m,x) = \sum_{v \in \partial \Gamma} \sum_{m=d(x,v)}^n \psi_K'(m,x,v) = \sum_{v \in \partial \Gamma} \psi_K(n,x,v).
\end{align*}

Our theorems above dealt with $\psi_K(x);$ the corollaries below provide estimates for some of these other quantities. These corollaries use the symbol $\approx$ from Definitions \ref{approx_def} and \ref{approx_def2} where constants are allowed both inside and outside exponentials. These constants depend on the constants appearing in the definitions of controlled weights, uniformly lazy, inner uniform, and Harnack graphs, and, in the case of uniform $S$-transience, on $L, \varepsilon.$ 

\begin{corollary}[Estimate on Hitting at a Point $v$]\label{hitv_notime}
Assume $\Gamma := \widehat{\Gamma} \setminus K$ is an inner uniform subgraph of a Harnack graph $(\widehat{\Gamma}, \pi, \mu).$ Then 
\begin{align}
\psi_K(x,v) \approx \sum_{\substack{y \in \Gamma: \\ y \sim v}} h(x)h(y) \pi(y) \sum_{n \geq d_\Gamma^2(x,y)} \frac{1}{V_h(x, \sqrt{n})} \quad \forall x \in \Gamma \setminus \partial_I \Gamma,\ v \in \partial \Gamma.
\end{align}

In the event that $\Gamma$ is uniformly $S$-transient, then
\begin{align}\label{hitK_atv}
\psi_K(x,v) \approx \sum_{n \geq d_\Gamma^2(x,v)} \frac{\pi(v)}{V_\Gamma(x,\sqrt{n})}.
\end{align}
\end{corollary}

\begin{proof}
Reasoning as in Theorem \ref{IU-psibd}, but without summing over \emph{all} points of the boundary of $\Gamma$ yields
\begin{align*}
\psi_K(x,v) &\approx \sum_{n \geq 2} \sum_{\substack{y \in \Gamma: \\ y \sim v}} \frac{h(x)h(y) \pi(y)}{V_h(x, \sqrt{n})} \exp\Big(-\frac{d_\Gamma^2(x,y)}{n}\Big)  \\
&\approx \sum_{\substack{y \in \Gamma: \\ y \sim v}} h(x)h(y) \pi(y) \sum_{n \geq d_\Gamma^2(x,y)} \frac{1}{V_h(x, \sqrt{n})}.
\end{align*}

If $\Gamma$ is uniformly transient, the result follows as $h\approx 1$ by Lemma \ref{his1}.
\end{proof}

\begin{remark}
In (\ref{hitK_atv}), a sum over the neighbors of $v$ that belong to $\partial_I \Gamma$ appears. For any $x,\, v$ there is always a point $y_{x,v} \in \partial_I \Gamma$ such that $d_\Gamma(x, y_{x,v}) +1 = d_\Gamma(x,v),$ but there may be multiple points that achieve this or other neighbors of $v$ that are further away from $x$ in $\Gamma.$ In the lower bound, we may keep only the point $y_{x,v},$ but, in the upper bound, we do not know a relationship that would allow us to replace a generic $h(y)$ by $h(y_{x,v}).$ If $h\approx 1,$ or if we know all $y\sim v$ are close in $\Gamma$ (not just in $\widehat{\Gamma}),$ this is not a problem and only one $y_{x,v}$ counts. However, if $v$ can be approached from multiple ``sides," this is not the case, and in fact $h$ may be very different on the different sides. (Consider a slit domain or two sides of a boundary with a ``corner.'') 

We are, however, always free to replace $\pi(y)$ by $\pi(v)$ due to the assumption of controlled weights.
\end{remark}

\begin{corollary}[Hitting at time $m$ at $v,$ rates of convergence]\label{hitprob_deriv}
Assume $\Gamma := \widehat{\Gamma} \setminus K$ is an inner uniform subgraph of a Harnack graph $(\widehat{\Gamma}, \pi, \mu).$ Then
\begin{align}\label{deriv_bd}
 \psi'_K(m,x,v) \approx \sum_{y \in \Gamma: y \sim v} \frac{h(x)h(y) \pi(y) }{V_h(x, \sqrt{m})} \exp\Big(-\frac{d_\Gamma^2(x,y)}{m}\Big) \quad \forall x \in \Gamma \setminus \partial_I \Gamma, v \in \partial \Gamma, m \geq d_\Gamma(x,v)
\end{align}
and for $n \geq d_\Gamma^2(x,v),$ 
\begin{align}\label{conv_est}
\psi_K(x,v) - \psi_K(n,x,v) \approx \sum_{y \in \Gamma: y \sim v} h(x)h(y)\pi(y) \sum_{m=n}^\infty \frac{1}{V_h(x, \sqrt{m})}
\end{align}

If $\Gamma$ is uniformly $S$-transient, then 
\begin{align} \psi'_K(m,x,v) \approx \frac{\pi(v)}{V_\Gamma(x,\sqrt{m})} \exp \Big(-\frac{d_\Gamma^2(x,v)}{m}\Big),\end{align}
and when $n \geq d_\Gamma^2(x,v),$ 
\begin{align}
\psi_K(x,v) - \psi_K(n,x,v) \approx \pi(v) \sum_{m=n}^\infty \frac{1}{V_\Gamma(x, \sqrt{m})}. 
\end{align}

\end{corollary}

\begin{proof}
Proceeding as in the proof of Theorem \ref{IU-psibd}, but summing in neither time nor space, for $x \in \Gamma \setminus \partial_I \Gamma, v \in \partial \Gamma,$ 
\begin{align*}
\psi'_K(m,x,v) &= \bP^x(X_{\tau_K}=v, \tau_K = m) = \sum_{y \in \Gamma: y \sim v} \mathcal{K}_{\Gamma,D}^{m-1}(x,y) \mathcal{K}_\Gamma(y, v) \\&\approx \sum_{y \in \Gamma: y \sim v} \frac{h(x) h(y) \pi(y_v)}{V_h(x,\sqrt{m})} \exp\Big(-\frac{d_\Gamma^2(x,y)}{m}\Big).
\end{align*}

Note points $y$ in the sum above will only appear if $m \geq d_\Gamma(x,y)$ (there is always at least one such $y$ since $m \geq d_\Gamma(x,v)$ by assumption). 

To obtain (\ref{conv_est}), use (\ref{deriv_bd}) to find
\begin{align*}
\psi_K(x,v) - \psi_K(n,x,v) &= \sum_{m=n}^\infty \psi_K'(m,x,v) \\&\approx \sum_{y \in \Gamma: y \sim v} h(x)h(y) \pi(y) \sum_{m=n}^\infty \frac{1}{V_h(x, \sqrt{m})} \exp\Big(-\frac{d_\Gamma^2(x,y)}{m}\Big).
\end{align*}
When $n \geq d_\Gamma^2(x,y),$ the exponential does not count. 

In the uniformly $S$-transient case, we know $h\approx 1.$ For the lower bounds, discard any inconvenient terms; for the upper bounds, recall the number of neighbors $y$ of $v$ is bounded above and all such neighbors satisfy $d_\Gamma(x,y) +1 \geq d_\Gamma(x,v) \geq 2.$ 
\end{proof} 

\begin{corollary}[Hitting at $v$ by time $n$]\label{hitv_beforen} 
Assume $\Gamma := \widehat{\Gamma} \setminus K$ is an inner uniform subgraph of a Harnack graph $(\widehat{\Gamma}, \pi, \mu).$ For all $x \in \Gamma \setminus \partial_I \Gamma, v \in \partial \Gamma, n \geq d_\Gamma(x,v),$
\begin{align}\label{hith_nxv}
\psi_K(n,x,v) \approx \sum_{y \in \Gamma: y \sim v} \bigg[ \frac{h(x)h(y) \pi(y) d_\Gamma^2(x,y)}{V_h(x, d_\Gamma(x,y))} \exp\Big(-\frac{d_\Gamma^2(x,y)}{n}\Big) + \sum_{m=d_\Gamma(x,y)^2}^n \frac{h(x)h(y)\pi(y)}{V_h(x, \sqrt{m})}\bigg]. 
\end{align}

In the uniformly $S$-transient case,
\begin{align}\label{hitprob_bd}
\psi_{K}(n,x,v)  &\approx \frac{\pi(v) d_\Gamma^2(x,v)}{V_\Gamma(x, d_\Gamma(x,v))} \exp\Big(-\frac{d_\Gamma^2(x,v)}{n}\Big) + \sum_{m=d_\Gamma(x,v)^2}^n \frac{\pi(v)}{V_\Gamma(x, \sqrt{m})}.
\end{align} 

\end{corollary}

\begin{proof}
This quantity is like that of Theorem \ref{IU-psibd}, except that the sum in time stops at a value $n$ instead of continuing to infinity. We are forced to consider several cases about the relationship between the size of $n$ and $d_\Gamma(x,v).$ As before, the uniformly $S$-transient case will follow by recalling $h\approx 1$ and that only one $y \sim v$ counts.  

In all cases, using Corollary \ref{hitprob_deriv}, 
\begin{align*}
\psi_K(n,x,v) &= \sum_{m=0}^n \psi'_K(m,x,v) 
\approx \sum_{\substack{y \in \Gamma: \\ y \sim v}}  \, \sum_{m=d_\Gamma(x,y)}^n \frac{h(x)h(y)\pi(y)}{V_h(x, \sqrt{m})} \exp\Big(-\frac{d_\Gamma^2(x,y)}{m}\Big).
\end{align*}

We now compute the inner sum in time $m$ above. For simplicity, we will often abbreviate $d_\Gamma(x,y)$ by $d$ in the rest of the proof. 

\noindent \underline{Case 1:} Total time $n$ is small compared to distance, that is $d_\Gamma(x,y) \approx n;$ say $d_\Gamma(x,y) \leq n \leq 2 d_\Gamma(x,y).$ 

Then the inner sum is roughly
\begin{align*}
 h(x)h(y)\pi(y) \sum_{m=d}^n \frac{1}{V_h(x, \sqrt{d})} \exp\Big(-\frac{d^2}{d}\Big) \approx \frac{h(x)h(y)\pi(y)}{V_h(x, d)} \exp(-d). 
\end{align*}
In this situation the exponential is very small, so any powers of $d$ that appear by taking the sum or adjusting the radius of $V_h$ can be fed to the exponential by changing the constant. Recall $V_h$ is doubling (see Theorem \ref{ph_Harnack}). 

\noindent \underline{Case 2:} The intermediate case, $2 d_{\Gamma}(x, y) \leq n < d_{\Gamma}^2(x, y)$.

We use a dyadic decomposition and cut the sum into pieces where $d2^{-l-1} \leq \sqrt{m} \leq d 2^{-l}.$ Recall we use $\asymp$ to denote such decomposition. Let $a$ denote the integer such that $\sqrt{n} \asymp d2^{-a},$ or $a \asymp \log_2(d/\sqrt{n})$ and $b$ be the integer such that $\sqrt{d} \asymp d2^{-b}$ or $b \asymp \log_2(\sqrt{d}).$ Since $d/\sqrt{n} \leq \sqrt{d}$ in this case, we have $a \leq b.$ Hence using the same tools to compute the sum as above, where the constants $C,c$ can change from line to line,
\begin{align*}
\sum_{m=d_\Gamma(x,y)}^n \frac{Ch(x)h(y)\pi(y)}{V_h(x, \sqrt{m})}& \exp\Big(-\frac{d_\Gamma^2(x,y)}{cm}\Big) \\
&\leq C h(x) h(y)\pi(y) \sum_{l=a}^b \sum_{\sqrt{m} \asymp d2^{-l}} \frac{1}{V_h(x, d2^{-l})} \exp\Big(-\frac{4^l}{c}\Big) \\
&\leq C h(x) h(y) \pi(y) \sum_{l=a}^b \frac{d^2}{4^l} \frac{1}{V_h(x, d2^{-l})} \exp\Big(-\frac{4^l}{c}\Big) \\
&\leq \frac{C h(x)h(y)\pi(y) d^2}{V_h(x, d)} \sum_{l=a}^b \exp\Big(-\frac{4^l}{c}\Big)\\
&\leq \frac{C h(x)h(y)\pi(y) d^2}{V_h(x, d)} \exp\Big(-\frac{4^a}{c}\Big).
\end{align*}
The last line follows from bounding the sum from above by $\sum_{l \geq a},$ and recalling $4^a \approx d^2/n$. For the lower bound, repeat the same series of steps, except in the last line keep only the first term $l=a.$ 

We have found 
\begin{align*}
\sum_{m=d_\Gamma(x,y)}^n \frac{h(x)h(y)\pi(y)}{V_h(x, \sqrt{m})} \exp\Big(-\frac{d_\Gamma^2(x,y)}{m}\Big) \approx \frac{h(x)h(y) \pi(y) d^2}{V_h(x,d)} \exp\Big(-\frac{d^2}{n}\Big).
\end{align*}

\noindent \underline{Case 3:} The case where time is large compared to the distance squared, $n \geq d_\Gamma^2(x,y)$.

Cut the sum into two pieces: where $m < d^2$ and where $m\geq d^2.$ For the first piece, apply the previous case. Here the exponential is large, so we may always ignore it. We find
\begin{align*}
\sum_{m=d_\Gamma(x,y)}^n \frac{h(x)h(y)\pi(y)}{V_h(x, \sqrt{m})} \exp\Big(-\frac{d_\Gamma^2(x,y)}{m}\Big) \approx \frac{h(x)h(y) \pi(y) d^2}{V_h(x,d)} + \sum_{m=d^2}^n \frac{h(x)h(y)\pi(y)}{V_h(x, \sqrt{m})}.
\end{align*} 

\noindent To finish estimating $\psi_K(n,x,v),$ take the sum in points $y.$ Different points $y\sim v$ may fall into different cases above, but in all cases the expression found matches that of (\ref{hith_nxv}). 
\end{proof}

\begin{corollary}[Hitting $K$ at time $m$]\label{hit_atm}
Assume $\Gamma := \widehat{\Gamma} \setminus K$ is an inner uniform subgraph of a Harnack graph $(\widehat{\Gamma}, \pi, \mu).$ For all $x \in \Gamma \setminus \partial_I \Gamma$ and all $m \geq d_\Gamma(x,K),$ 
\begin{align}
\psi_K'(m,x) \approx \sum_{\substack{y \in \partial_I \Gamma \\ d_\Gamma(x,y) \leq m}} \frac{h(x) h(y) \pi(y)}{V_h(x, \sqrt{m})} \exp\Big(-\frac{-d_\Gamma^2(x,y)}{m}\Big).
\end{align}

If $\Gamma$ is uniformly $S$-transient, then 
\begin{align}
\psi_K'(m,x) \approx \sum_{\substack{y \in \partial_I \Gamma \\ d_\Gamma(x,y) \leq m}} \frac{\pi(y)}{V_\Gamma(x, \sqrt{m})} \exp\Big(-\frac{d_\Gamma^2(x,y)}{m}\Big). 
\end{align}
\end{corollary}

The corollary follows from similar arguments as above. There are no particularly nice simplifications for any of these expressions since the sums in space rely on $h, \pi,$ and $d_\Gamma(x,y).$ 

Note either of the previous two corollaries could be used to get an estimate on $\psi_K(n,x).$ 

\subsection{Examples}

In this section we apply the results of previous sections to various examples.  

Recall we have already seen that $\Z^m \setminus \Z^k$ is uniformly $S$-transient when $k \leq m -3.$ Further, it is not too difficult to verify that $\Z^m \setminus \Z^k$ is uniform if and only if $k \leq m-2,$ so that the results of the previous section also apply. This example generalizes as follows. 

\begin{example}[Examples with regular volume growth]

Let $\Gamma := \widehat{\Gamma} \setminus K$ be inner uniform inside the Harnack graph $(\widehat{\Gamma}, \pi, \mu).$ Assume there exists $\alpha>0$ such that $V_\Gamma(x,r) \approx r^\alpha$ for all $x \in \Gamma, r>0.$ Further assume that $V_{\partial_I \Gamma}$ is doubling in the sense of Definition \ref{bdry_double} and that $\partial_I \Gamma$ is regular in the sense that there exists $\beta >0$ such that $V_{\partial_I \Gamma}(y,r) \approx r^\beta$ for all $y \in \partial_I \Gamma, r>0.$ Assume $\alpha - \beta > 2.$ 

Then we may use Corollary \ref{simpler} to justify that $\Gamma$ is uniformly $S$-transient, in which case $h \approx 1$ and (\ref{2sided_trans}) gives us a two-sided bound on $\psi_K$ as a function of $x$:
\begin{align*}
\psi_K(x) \approx \sum_{n \geq d_\Gamma^2(x, \partial_I \Gamma)} \frac{V_{\partial_I \Gamma}(y_x, \sqrt{n})}{V_\Gamma(x,\sqrt{n})} \approx \sum_{n \geq d_\Gamma^2(x, \partial_I \Gamma)} \frac{1}{n^{(\alpha-\beta)/2} } \approx \frac{1}{d_\Gamma(x, \partial_I \Gamma)^{\alpha-\beta-2}}. 
\end{align*} 
\end{example}

\begin{example}[Half-space, $\Z^m \setminus \Z^{m-1}$]\label{halfspace_main}
Consider upper half-space $\Gamma = \Z^m_+ = \{ (x_1, \dots, \allowbreak x_m) \in \Z^d: x_m >0\}$ inside of $\Z^m$ with the lazy simple random walk. Let $\vec{x}=(x_1, \dots, x_m) \in \Gamma.$ We consider the chance a walk hits $\vec{v} = (v_1, \dots, v_{m-1}, 0)$ from $\vec{x}.$ Clearly $\Gamma$ is inner uniform in $\Z^m,$ which is Harnack. In this case, $h(\vec{x}) = x_m.$ Let $\vec{y}_v := (v_1, \dots, v_{m-1},1).$ 

Let $x = (x_1, \dots, x_{m-1}), \ v = (v_1, \dots, v_{m-1})$ and $d(x,v) = |x_1 - v_1| + \cdots + |x_{m-1} - v_{m-1}|.$ Applying various corollaries from the previous section (and assuming $n \geq d_\Gamma^2(\vec{x}, \vec{v})$ where sensible), 
\begin{align*}
& \psi_K(\vec{x},\vec{v}) \approx \frac{x_m}{[d(\vec{x}, \vec{y}_v)]^m} = \frac{x_m}{(d(x,v) +|x_m -1|)^m} \\[1ex]
&\psi_K'(n, \vec{x}, \vec{v}) \approx \frac{x_m}{(x_m + \sqrt{n})^2 \, n^{m/2}} \exp\Big(-\frac{([d(x,v)]^2 + |x_m|^2)}{n}\Big)\\[1ex]
& \psi_K(\vec{x}, \vec{v}) - \psi_K(n, \vec{x}, \vec{v}) \approx \frac{x_m}{n^{m/2}} \\[1ex]
& \psi_K(n, \vec{x}, \vec{v}) \approx \frac{x_m}{[d_\Gamma(\vec{x}, \vec{v})]^m} \exp\Big(-\frac{d_\Gamma^2(\vec{x}, \vec{v})}{n}\Big) + x_m \Big[ \frac{1}{[d_\Gamma(\vec{x}, \vec{v})]^m} - \frac{1}{n^{m/2}}\Big] \\[1ex]
&\psi'_K(n,\vec{x}) \approx \frac{x_m}{n^{3/2}}\\[1ex]
&\psi_K(\vec{x}) - \psi_K(n, \vec{x}) \approx \frac{x_m}{n^{1/2}} \\[1ex]
&\psi_K(n,\vec{x}) \approx x_m \bigg[ \frac{1}{x_m^{1/2}} - \frac{1}{n^{1/2}}\bigg]. 
\end{align*}

The above estimate for $\psi_K(\vec{x}, \vec{v})$ is essentially a (multivariate) Cauchy distribution as expected. This is clearer to see if we take $m=2, \, \vec{x}=(0,2),$ and $\vec{v} = (v,0)$ so that $\psi_K(\vec{x},\vec{v}) \approx \frac{1}{(1+|v|)^2} \approx \frac{1}{1+v^2}.$ Further, notice that the rate of convergence of $\psi_K(n, \vec{x}, \vec{v})$ in time to $\psi_K(\vec{x}, \vec{v})$ is dependent on the dimension, but convergence of $\psi_K(n, \vec{x})$ to $\psi_K(\vec{x}) \equiv 1$ has the same rate in all dimensions. 

\end{example}

\begin{example}[Cones in $\Z^2$] Let $\widehat{\Gamma} = \Z^2$ and $\Gamma$ be the lattice points lying inside of a cone of aperture $\alpha \in (0,2\pi)$ with vertex at $(0,0)$ and one side of the cone lying along the $x$-axis. Note this is a case where $K \not = \partial \Gamma.$ As the cone is inner uniform, the results of the previous section apply.

In the continuous case, it is known that the profile of such a cone is $h(r, \theta) = r^{\pi/\alpha} \sin\big(\frac{\pi}{\alpha} (\theta)\big),$ where $(r, \theta) \in \R^2$ are polar coordinates. Since a cone can be thought of as the graph above a Lipschitz domain, a result of Varopoulos \cite{Varo_LipCLT} says harmonic functions in the discrete (lattice) and continuous versions of a space are similar away from the boundary. For further discussion of harmonic functions in cones see \cite{DW_altharmcone} and references therein. 

Therefore, assuming $\vec{x} \in \Z^2$ is away from the boundary of our discrete cone and $\vec{v} \in \Z^2$ lies along the boundary of the cone, by Corollary \ref{hitv_notime},
\begin{align*}
\psi_K(\vec{x}, \vec{v}) \approx \frac{|\vec{x}|^{\pi/\alpha} \sin\big(\frac{\pi}{\alpha}(\theta_{\vec{x}})\big) |\vec{y}_{\vec{v}}|^{\pi/\alpha}\sin\big(\frac{\pi}{\alpha}(\theta_{\vec{y}_{\vec{v}}})\big)}{[d_\Gamma(\vec{x}, \vec{y}_{\vec{v}})]^{2\pi/\alpha}},
\end{align*}
where $\vec{y}_{\vec{v}} \sim \vec{v}$ and belongs to $\Gamma.$ One can verify this result matches that of the half-plane in the previous example ($\alpha = \pi, m =2$). 

We can also express $\psi_K(\vec{x}, \vec{v})$ in terms of distances to the edges of the cone. Let the edge of the cone that lies along the $x$-axis be $L_0$ and the other edge be $L_1.$ Then $|\vec{x}| \approx d(\vec{x}, L_0) + d(\vec{x}, L_1)$ and, for $\alpha$ fixed, $\sin\big(\frac{\pi}{\alpha}(\theta_{\vec{x}})\big) \approx \frac{d(\vec{x}, L_0)}{|\vec{x}|}\frac{d(\vec{x}, L_1)}{|\vec{x}|}.$ (Note one of these factors is always roughly constant.) Thus
\begin{align*}
\psi_K(\vec{x}, \vec{v}) \approx \frac{[d(\vec{x}, L_0) + d(\vec{x}, L_1)]^{\frac{\pi}{\alpha} - 2} \,  [d(\vec{v}, L_0)+d(\vec{v}, L_1)]^{\frac{\pi}{\alpha}-1} \, d(x, L_0), d(x, L_1)}{[d_\Gamma(\vec{x}, \vec{y}_{\vec{v}})]^{2\pi/\alpha}}.
\end{align*}

\end{example}

\begin{example}[A line in $\Z^3$, e.g. $\Z^m \setminus \Z^{m-2}$]\label{Z3_line} Consider $\widehat{\Gamma} = \Z^3$ and $K=\{(0,0,x_3): x_3 \in \Z\},$ the $x_3$-axis. The arguments below apply more generally to $\Z^m \setminus \Z^{m-2}.$ The harmonic profile is the same as the harmonic profile of a single point in $\Z^2,$ and consequently $h(x_1,x_2,x_3) \approx \log (|x_1| +|x_2|+1)$ (see e.g. \cite[Section 11]{Spitzer}). Given $\vec{x} = (x_1, x_2, x_3) \in \Gamma := \Z^3 \setminus K,$ then $d_{\vec{x}} := d(\vec{x}, \partial_I \Gamma) = |x_1| + |x_2|-1.$ We can use Theorem \ref{IU-psibd} to check that $\psi_K(\vec{x}) \approx 1$:
\begin{align*}
\psi_K(\vec{x}) &\approx \sum_{n \geq d_{\vec{x}}^2} \sum_{\substack{\vec{y} \in \partial_I \Gamma, \\ d_\Gamma^2 (\vec{x},\vec{y}) \leq n}} \frac{h(\vec{x})h(\vec{y})}{h(\vec{x}_{\sqrt{n}}) h(\vec{y}_{\sqrt{n}})} \frac{\pi(\vec{y})}{V_\Gamma(\vec{x}, \sqrt{n})} \\
&\approx \sum_{n \geq d_{\vec{x}}^2} \ \sum_{\vec{y}: d_\Gamma^2(\vec{x},\vec{y})\leq n} \frac{\log(d_{\vec{x}})}{\log(d_{\vec{x}} + \sqrt{n}) \log(\sqrt{n}) n^{3/2}} \\
&\approx \sum_{n \geq d_{\vec{x}}^2}  \frac{\log(d_{\vec{x}})}{(\log(n))^2 n} \approx \frac{\log(d_{\vec{x}})}{\log(d_{\vec{x}})} = 1.
\end{align*}

The above calculation used that he number of $\vec{y}\,$'s in the $x_3$-axis at distance less than $\sqrt{n}$ from $\vec{x}$ is about $\sqrt{n} - d_{\vec{x}} \approx \sqrt{n}.$ This is sensible if we replace the exterior sum $n\geq d_{\vec{x}}^2$ by $n \geq c d_{\vec{x}}^2$; for the lower bound, we can throw this away, and, in the upper bound, the sum over $d_{\vec{x}}^2 \leq n \leq c d_{\vec{x}}^2$ is can be controlled by later pieces of the sum. This might seem simple, but the fact that we can make manipulations like this in our calculations relies on the fact the boundary is doubling. (See Remark \ref{sums} below.) 

It is more interesting to compute $\psi_K(\vec{x}, \vec{v})$ where $\vec{v}=(0,0,v) \in K$. Then via Corollary \ref{hitv_notime}:
\begin{align*}
\psi_K(\vec{x}, \vec{v}) &\approx \sum_{\vec{y} \in \Gamma: \vec{y} \sim \vec{v}} h(\vec{x})h(\vec{y}) \pi(w) \sum_{n \geq d_\Gamma^2(\vec{x},\vec{y})} \frac{1}{V_h(\vec{x}, \sqrt{n})}  
\approx \log(d_{\vec{x}}) \sum_{n \geq d_\Gamma^2(\vec{x},\vec{y})} \frac{1}{(\log(n))^2 n^{3/2}} \\
&\approx \frac{\log(d(\vec{x}, \vec{v}_{\vec{x}}))}{(d_\Gamma(\vec{x}, \vec{v}_{\vec{x}}) + d_\Gamma(\vec{v}_{\vec{x}}, \vec{v})) (\log(d_\Gamma(\vec{x},\vec{v}_{\vec{x}}) + d_\Gamma(\vec{v}_{\vec{x}}, \vec{v}))^2}. 
\end{align*}
\end{example}

\begin{remark}\label{sums}
Given $x \in \Gamma, \ y \in \partial_I \Gamma,$ it is always true that $d_\Gamma(x,y) \approx d_\Gamma(x, y_x) + d_\Gamma(y_x, y),$ where $y_x \in \partial_I \Gamma$ achieves $d_\Gamma(x, \partial_I \Gamma)$ (and that $y_x \sim v_x \in \partial \Gamma$ that achieves $d(x, \partial \Gamma)).$  Provided changing the ``radius" by a constant does not really change how many points $y \in \partial_I \Gamma$ are at a particular distance from $y_x \in \partial_I \Gamma,$ then when $n$ is sufficiently large, the inner sums in our theorems/corollaries can be taken over $y \in \partial_I \Gamma: d^2(y, y_x) \leq n.$ This remark is similar in spirit to Corollary \ref{simpler}; Example \ref{fly} below gives an example where such assumptions do not hold. 
\end{remark}

\begin{example}[Weighted half-spaces]

This example is a continuation of Example \ref{halfspace_trans}. Once again we consider $\Gamma = \{ \vec{x} = (x_1, \dots, x_m) \in \Z^m : x_m >0\}$ inside  $\Z^m_{\geq 0}$ with weight $(1+x_m)^\alpha.$ Provided $\alpha > -m,$ then $\Z^m_{\geq 0 }$ with this weight is Harnack, which can be shown using similar arguments to those given in Section 4.3 of \cite{ag_lsc_harnackstability}. The profile for such a space clearly only depends on the $x_m$ coordinate and reduces to computing the profile on the weighted half-line. Using the definition of harmonic and choosing the scaling by setting $h(x_1, \dots, x_{m_1}, 1) =1,$ we can compute 
\[ h(x_1, \dots, x_m) = \begin{cases} \sum_{l=1}^{x_m} \frac{1}{l^\alpha}, & \alpha \geq0 \\[1ex] \sum_{n=2}^{x_m+1} \frac{2^\alpha}{n^\alpha}, & \alpha \in (-N,0) \end{cases} \quad \approx x_m^{1-\alpha}.\]

If $\alpha >1,$ then $1-\alpha <0$ and it is clear $h$ is uniformly bounded above and below. In Example \ref{halfspace_trans}, we already saw that $\Z^m_+$ was uniformly $S$-transient with such weights. Using Theorem \ref{IU-psibd} gives us a lower bound that matches the upper bound found in Example \ref{halfspace_trans}, and we can also find $\psi_K(\vec{x}, \vec{v})$:
\begin{align*}
 &\psi_K(\vec{x}) \approx \frac{1}{x_m^{\alpha -1}}\\
 &\psi_K(\vec{x}, \vec{v}) \approx \frac{1}{[d(\vec{x}, \vec{v})]^{m+\alpha-2}} \approx \frac{1}{[|x_1 - v_1| + \cdots + |x_{m-1} - v_{m-1}| + |x_m|]^{m+\alpha-2}}. 
\end{align*}

Now consider $\alpha \in (-N, 1].$ Using our Theorem \ref{IU-psibd}, we find that $\psi_K(\vec{x})$ is roughly constant and can compute $\psi_K(\vec{x}, \vec{v}):$ 
\begin{align*}
&\psi_K(\vec{x}) \approx \sum_{n \geq d^2(\vec{x}, \partial \Gamma)} \frac{x_m^{1-\alpha} n^{(m-1)/2}}{n^{1-\alpha} n^{(m+\alpha)/2}} \approx \sum_{n \geq x_m^2} \frac{x_m^{1-\alpha}}{n^{(3-\alpha)/2}} \approx \frac{x_m^{1-\alpha}}{x_m^{1-\alpha}} \approx 1\\
&\psi_K(\vec{x}, \vec{v}) \approx \frac{x_m^{1-\alpha}}{[d(\vec{x}, \vec{v})]^{m-\alpha}} = \frac{x_m^{1-\alpha}}{[|x_1-v_1| + \cdots + |x_{m-1}-v_{m-1}| + |x_m|]^{m-\alpha}}. 
\end{align*} 

Substituting $\alpha = 0$ into the expression for $\psi_K(\vec{x}, \vec{v})$ above, we recover the formula from Example \ref{halfspace_main}. 

In general knowing $\psi_K(\vec{x}) \approx 1$ is not sufficient to conclude a subgraph is $S$-recurrent, as this does not necessarily imply $\psi_K(\vec{x}) = 1.$ However, in this specific case, we can use symmetry to argue that the half-space cannot be $S$-recurrent and have $\psi_K(\vec{x})$ uniformly bounded below away from zero. Clearly, $\psi_K(\vec{x})$ only depends on $x_m = d(\vec{x}, \partial \Gamma).$ Also, by using repeated applications of the Markov property, if $x_m = d,$ then in order for the random walk to hit the set $\{x_m=0\},$ it must first hit the set $\{x_m = d-1\},$ then the set $\{x_m = d-2\},$ and so on, so the probability of hitting $\{x_m=0\}$ decomposes into a product of probabilities of hitting a set that is distance $1$ away from the starting point. Although the weights are different if we consider hitting $\{x_m=0\}$ from a point where $x_m=1$ in the usual half-space versus hitting $\{x_m = k\}$ from a point where $x_m = k+1$ in the half-space $\{x_m \geq k\},$ the weights will be uniformly comparable. Since $\psi_K$ is the chance of hitting $K$ before time $\infty,$ a bounded change of weights will not change it. Hence if $\psi_K(\vec{x}) <1$ everywhere, there must be points where $\psi_K$ is arbitrarily close to zero. Hence knowing $\psi_K(\vec{x}) \approx 1$ shows that these weighted half-spaces are in fact $S$-recurrent. 

\end{example}

\begin{example}[``Flyswatter'']\label{fly}
In $\Z^4,$ consider $K$ to be a two-dimensional infinite ``flyswatter" as in Figure \ref{flyswatter} below. A key point is that the flyswatter has long ``handles" and ``mesh parts" at every scale; this causes $K= \partial \Gamma$ to fail to be doubling in $\Z^4.$ However, $\Gamma = \Z^4 \setminus K$ is uniform as one can always use the extra two dimensions to move away from the flyswatter, and $d_\Gamma \approx d_{\Z^4}$ since the flyswatter is either thin or has frequent holes. While Theorem \ref{IU-psibd} and associated corollaries apply to this example, we do not know how to compute $h.$ This situation is typical. 

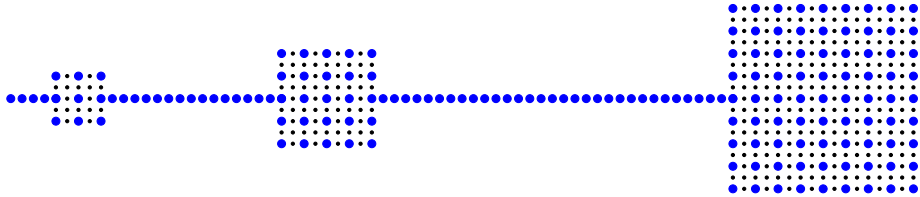
\begin{figure}[]
\begin{center}
\begin{tikzpicture}[scale=.3]
        \foreach \i in {12,12.5,13, 13.5,...,16}
        \foreach \j in {-2,-1.5,-1,...,2}
         \filldraw (\i, \j) circle (2pt); 
         \foreach \i in {32, 32.5,...,40}
          \foreach \j in {-4,-3.5,-3,...,4}
           \filldraw (\i, \j) circle (2pt); 
	\foreach \i in {0,0.5,1,1.5,2,4.5,5,5.5,6,6.5,7,7.5,8,8.5,9,9.5,10,10.5,11,11.5,12}
	 \filldraw[blue] (\i, 0) circle (5pt);
	 \foreach \i in {16.5,17,17.5,18,...,32}
	  \filldraw[blue] (\i, 0) circle (5pt);
	 \foreach \i in {2,3,4}
	  \foreach \j in {-1,0,1}
	   \filldraw[blue] (\i, \j) circle (5pt);
	   \foreach \i in {12,13,14,15,16} 
	    \foreach \j in {-2,-1,0,1,2}
	     \filldraw[blue] (\i, \j) circle (5pt); 
	     \foreach \i in {32,33,34,35,36,37,38,39,40}
	      \foreach \j in {-4,-3,-2,-1,0,1,2,3,4}
	      \filldraw[blue] (\i,\j) circle(5pt);
	  \foreach \i in {2.5,3.5}
	  \foreach \j in {-1,-0.5,0,0.5,1}
	   \filldraw (\i, \j) circle (2pt);
	   \foreach \i in {2,3,4}
	    \foreach \j in {-0.5, 0.5}
	     \filldraw(\i, \j) circle (2pt);
\end{tikzpicture}
\end{center}
\caption{The blue ``flyswatter," which we imagine continues infinitely in both directions in a similar manner. Although this picture is in two dimensions, we think of this in a higher dimensional space. Note how there are black points in-between the blue points, and it is easy to see distance in $\Z^d$ would not be changed significantly by avoiding the blue points when $d \geq 4.$}\label{flyswatter}
\end{figure}

\end{example} 

\subsection{Example: A set that is \texorpdfstring{$S$}{S}-transient but not uniformly so}\label{para_Z4}

In this section, we discuss an example that turns out to be $S$-transient, but not uniformly so, illustrating the distinction between these notions. We apply both Theorems \ref{psi-upperbd} and \ref{IU-psibd} and discuss what we can say about its harmonic profile $h$.

Let $\widehat{\Gamma} = \Z^4.$ Think of $\mathbf{x} \in \Z^4$ as $\mathbf{x}=(x_1, x_2, x_3, x_4).$ In the $x_1x_2$-plane, let $K = \partial \Gamma$ be the set of lattice points that lie inside the graph of $x_2 = \pm x_1^\alpha$ for $\alpha \in (0,1), \ x_1 \in \Z_{\geq 0}.$ In the case $\alpha = 1/2,$ we have a parabola whose axis of symmetry is the $x_1$-axis; we may often refer to the points of $K$ as a ``parabola" regardless of the value of $\alpha$ (or the fact that we are only considering a discrete analog of a parabola). Note that $K$ is a two-dimensional object in four-dimensional space, so $\Gamma := \Z^4 \setminus K$ is inner uniform.

If we consider the lazy simple random walk on $\Z^4,$ then it has controlled weights, is uniformly lazy, and is Harnack. Hence we can apply any of our results to this example. 

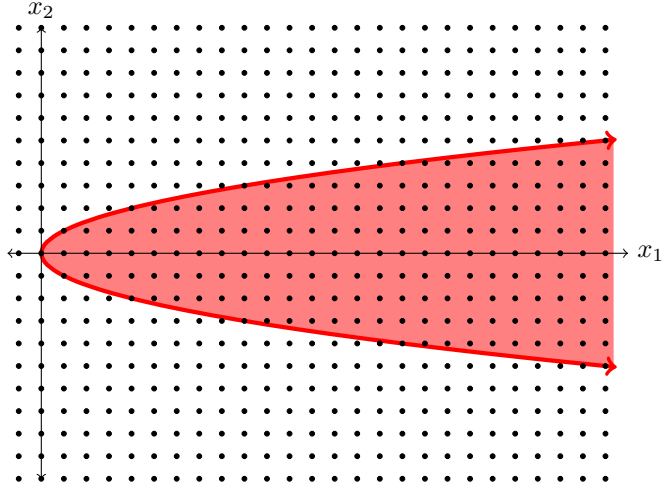
\begin{figure}[h]
\centering
\begin{tikzpicture}[scale=.3]
  \draw[fill=red!50, domain=-5.05:5.05, smooth, variable=\y, draw=red, <->,ultra thick]  plot ({\y*\y}, {\y});
      \foreach \i in {-1,...,25}
      \foreach \j in {-10,...,10}
        \filldraw (\i,\j) circle(3pt);
  \draw[<->] (-1.5, 0) -- (26, 0) node[right] {$x_1$};
  \draw[<->] (0, -10) -- (0, 10) node[above] {$x_2$};
\end{tikzpicture}
 \caption{For $\alpha=1/2,$ we take the lattice points inside of the parabola $x_1=x_2^2$ as our set $K.$ This figure is the $x_1x_2$-plane that lives inside $\Z^4.$}
  \end{figure}

We first use Theorem \ref{psi-upperbd} to show that $\Z^4$ is $S$-transient with respect to $K$. Doubling of traces of balls in $\partial \Gamma$ can be seen by the following formula for $V_{\partial \Gamma}:$ 
\[  V_{\partial \Gamma}(\mathbf{x}, r) \approx \begin{cases} r^2, & r \leq |x_1|^\alpha \\ |x_1|^\alpha r, & |x_1|^\alpha < r < |x_1| \\ r^{\alpha+1}, & r \geq |x_1| \end{cases} \quad \text{ for } \mathbf{x}=(x_1,x_2,0,0) \in K.\]

For any point $\mathbf{x} \in \widehat{\Gamma},$ we have $V_{\widehat{\Gamma}}(\mathbf{x}, r) \approx r^4.$ For any $\mathbf{x} \in \Gamma := \Z^4 \setminus K,$ let $\mathbf{x}^* = (x_1^*, x_2^*, x_3^*, x_4^*)$ denote the unique point in $K$ that achieves $d(\mathbf{x}, K).$ 
Thus for any $r>0$ and any $\mathbf{x} \in \Gamma,$
\begin{align*}
\widetilde{W}(\mathbf{x}, r) \approx \begin{cases} r^2, & r \leq |x_1^*|^\alpha \\ \frac{r^3}{|x_1^*|^\alpha}, & |x_1^*|^\alpha < r < |x_1^*| \\ r^{3-\alpha}, & r \geq |x_1^*|. \end{cases}
\end{align*}

\begin{lemma}\label{parabola_trans} 
The graph $\Z^4$ is $S$-transient with respect to the parabola $K.$ Moreover, for a given $\varepsilon >0,$ we can pick $L_\varepsilon = L$ sufficiently large so that in the regime where $d_{\mathbf{x}} := d_{\Z^4}(\mathbf{x},K) \geq |x_1^*| \geq L,$ we have $\psi_K(\mathbf{x}) \leq 1 - \varepsilon.$
\end{lemma}

\begin{proof}
Recall $\mathbf{x}^* = (x_1^*, x_2^*, x_3^*, x_4^*)$ is the point that achieves $d_{\Z^4}(\mathbf{x}, K).$ When $d_{\mathbf{x}} \geq |x_1^*|,$ 
\[ \psi_K(\mathbf{x}) \leq \sum_{n \geq d_{\mathbf{x}}^2} \frac{C}{\widetilde{W}(\mathbf{x}, \sqrt{n})} \approx \sum_{n \geq d_{\mathbf{x}}^2} \frac{1}{n^{(3-\alpha)/2}} \approx -\frac{1}{t^{(1-\alpha)/2}} \Big|_{t=d_{\mathbf{x}}^2}^\infty \approx \frac{1}{d_{\mathbf{x}}^{1-\alpha}} \leq \frac{1}{L^{1-\alpha}}. \]
Thus $\psi_K(\mathbf{x}) < 1- \varepsilon$ for $L_\varepsilon$ sufficiently large.  
\end{proof}

A key component of the proof of the above lemma was the assumption that $d_{\mathbf{x}} \geq |x_1^*|.$ If instead $|x_1^*|^\alpha \leq d_{\mathbf{x}} \leq |x_1^*|,$ then we have the bound
\begin{align}\label{parabola_middlecase}
\psi_K(\mathbf{x}) &\leq \sum_{n \geq  d_{\mathbf{x}}^2} \frac{C}{\widetilde{W}(\mathbf{x}, \sqrt{n})} = \sum_{n = d_{\mathbf{x}}^2}^{|x_1^*|^2} \frac{C |x_1^*|^\alpha}{n^{3/2}} + \sum_{n \geq |x_1^*|^2} \frac{C}{n^{(3-\alpha)/2}} \\[1ex]
&= \frac{c_a |x_1^*|^\alpha}{d_{\mathbf{x}}} - \frac{c_a}{|x_1^*|^{1-\alpha}} + \frac{c_b}{|x_1^*|^{1-\alpha}} ,
\end{align}
where the constants $c_a, c_b$ depend on the approximation of $\widetilde{W}$ and on the estimation of the sums above. We write these constants to emphasize that the $|x_1^*|^{\alpha-1}$ terms do not cancel. If instead $d_{\mathbf{x}} < |x_1^*|^\alpha,$ then there is a third term appearing in the estimate for $\psi_K$ given by Theorem \ref{psi-upperbd}/Corollary \ref{simpler}. 

Lemma \ref{parabola_trans} does not show that $\Z^4$ is uniformly $S$-transient with respect to the parabola $K$ since $d_{\mathbf{x}}$ and $|x_1^*|$ are related. Indeed, it is possible to pick a sequence of points $\{\mathbf{x}^m\}_{m\geq0}$ such that $d_{\mathbf{x}^m} \to \infty,$ but the bound in Theorem \ref{psi-upperbd} does not give useful information. To that end, consider points $\mathbf{x}^m = (x_1^m, x_2^m, x_3^m, x_4^m)$ that lie directly above the parabola so that $d_{\mathbf{x}^m} \approx x_3^m + x_4^m,$ which is independent of $x_1^m = (x_1^m)^*.$ Further, take $d_{\mathbf{x}^m} = |(x_1^m)^*|^\alpha$ for all $m.$ Then we are in the situation of (\ref{parabola_middlecase}) so that 
\begin{align*}
\psi_K(\mathbf{x}^m) \leq
\frac{c_a |(x_1^m)^*|^\alpha}{d_{\mathbf{x}^m}} - \frac{c_a}{|(x_1^m)^*|^{1-\alpha}} + \frac{c_b}{|(x_1^m)^*|^{1-\alpha}}  = c_a - \frac{c_a}{|(x_1^m)^*|^{1-\alpha}} + \frac{c_b}{|(x_1^m)^*|^{1-\alpha}}.
\end{align*}
Thus $\psi_K(\mathbf{x}^m) \to c_a$ as $d_{\mathbf{x}^m} = |(x_1^m)^*|^\alpha \to \infty.$ From this, we cannot conclude that $\psi_K(\mathbf{x}^m)$ tends to zero as $d_{\mathbf{x}^m} \to \infty$, and if $c_a \geq 1,$ this tells us no information on $\psi_K$ at all. Indeed, the appearance of the constant $c_a$ (essentially ``$1$") in the computation of the above sum indicates that Theorem \ref{psi-upperbd} will not give a useful bound in this regime. 

From Lemma \ref{parabola_trans}, we know that $\Gamma$ is $S$-transient and that  $h \approx 1$ in the region where $d(\mathbf{x}, K) \gg |x_1^*|$, since $\psi_K(\mathbf{x}) \leq 1- \varepsilon$ in this region. The two lemmas below capture how the results of Section \ref{h-trans} can improve our knowledge of $h$ as we approach the parabola in certain ways. 

\begin{lemma}\label{h_1_more}
For any $\mathbf{x} \in \Gamma$ satisfying $d_{\mathbf{x}} \gg |x_1^*|^\alpha \geq L,$ we have $h(\mathbf{x}) \approx 1.$ 
\end{lemma}

\begin{lemma}\label{h_para_bdry}
Let $\mathbf{u}^* = (u_1, 0, 0, 0) \in K$ and $B = B_{\widehat{\Gamma}}(\mathbf{u}^*, \frac{1}{2} |u_1|^\alpha).$ Then there exists a constant $0<a<1$ such that 
\[ h(\mathbf{x}) \approx c \, \frac{\log(d_{\mathbf{x}})}{\log(\hat{c} |u_1|^\alpha)} \quad \forall \mathbf{x} \in B_\Gamma( \mathbf{u}^*, a |u_1|^\alpha). \]
\end{lemma} 

\begin{proof}[Proof of Lemma \ref{h_1_more}]
We already know this result for $\mathbf{x} \in \Gamma$ satisfying $d_{\mathbf{x}} \gg |x_1^*|$ due to Lemmas \ref{parabola_trans} and \ref{his1}. Therefore it suffices to consider $\mathbf{x} \in \Gamma$ such that $d_\mathbf{x} \approx \hat{c} |x_1^*|^\alpha$ for some constant $\hat{c}.$ In this region, by (\ref{parabola_middlecase}) and Lemma \ref{his1},
\begin{align*}
h(\mathbf{x}) = 1 -\psi_K(u) \geq 1 + \frac{c_a - c_b}{|x_1^*|^{1-\alpha}} - c_a \frac{|x_1^*|^\alpha}{d_{\mathbf{x}}}. 
\end{align*}
If $c_a > c_b$ we can ignore the middle term; otherwise, assume $|x_1^*|^\alpha \geq L$ where $L$ is large enough to ensure $(c_a-c_b)/|x_1^*|^{1-\alpha} \geq - 1/4.$ Also choose $\hat{c}$ so that $c_a/\hat{c} < 1/4.$ With these choices, for $\mathbf{x}$ satisfying $d_{\mathbf{x}} \approx |x_1^*|^\alpha \geq L,$
\[ h(\mathbf{x}) \geq \frac{3}{4} - \frac{c_a |x_1^*|^\alpha}{\hat{c} |x_1^*|^\alpha} \geq \frac{1}{2}.\]
Thus $h\approx 1$ whenever $d_{\mathbf{x}} \gg |x_1^*|^\alpha \geq L$ as desired. 

\end{proof}

\begin{proof}[Proof of Lemma \ref{h_para_bdry}]
As this result is about points near the parabola $K,$ we use the boundary Harnack inequality. Given $\mathbf{u}^* = (u_1, 0, 0, 0) \in K,$ take $\mathbf{u} = (u_1, 0, u_3, u_4)$ such that $d_{\mathbf{u}} \approx |u_1|^\alpha$ and $h(\mathbf{u}) \approx 1$ as in Lemma \ref{h_1_more}. Note $\mathbf{u^*}$ is the projection of $\mathbf{u}$ onto $K.$ 

As $\widehat{\Gamma}$ is Harnack and $h$ is harmonic inside $\Gamma,$ by applying the elliptic Harnack inequality a finite number of times, we find a point (which we continue to call $\mathbf{u}$) such that $h(\mathbf{u}) \approx 1$ and $\mathbf{u}$ lies in $B = B_{\widehat{\Gamma}}(\mathbf{u}^*, \frac{1}{2} |u_1|^\alpha)$.

From the perspective of $B,$ we cannot tell that $K$ is not the entire $x_1x_2$-plane. As in $B$ we are looking at a two-dimensional ball inside of four-dimensional space, we know there is a positive harmonic function $f$ in $B$ that is zero on the intersection of $B$ with $K$ such that $f(\mathbf{x}) \approx \log (|x_3|^2 + |x_4|^2) \approx \log(d(\mathbf{x}, K)^2).$ 

Therefore, by the boundary Harnack inequality (Theorem \ref{bdry_Harnack}), 
\begin{align*}
\frac{f(\mathbf{x})}{f(\mathbf{u})} \leq A_1 \frac{h(\mathbf{x})}{h(\mathbf{u})} \implies c \, \frac{\log(d_{\mathbf{x}}^2)}{\log(d_{\mathbf{u}}^2)} = c \, \frac{\log(d_{\mathbf{x}})}{\log(\hat{c} |u_1|^\alpha)} \leq h(\mathbf{x}) \quad \forall \mathbf{x} \in B(\mathbf{u}^*, \frac{1}{2A_0}|u_1|^\alpha). 
\end{align*}
As we may also apply boundary Harnack in the other direction, we conclude 
\[ h(\mathbf{x}) \approx c \, \frac{\log(d_{\mathbf{x}})}{\log(\hat{c} |u_1|^\alpha)} \]
on a ball of radius strictly smaller than that of $B$ (but comparable to $|u_1|^\alpha$). 

\end{proof} 

The two lemmas above give a wide region where we understand $h$. However, there is still a ``bad region" of points where the behavior of $h$ remains unknown. For any $\mathbf{x} \in \Gamma,$ recall $\mathbf{x}^* = (x_1^*, x_2^*, x_3^*, x_4^*)$ is its projection on to $K.$ The behavior of $h$ is not known for points $\mathbf{x} \in \Gamma$ where $d_{\mathbf{x}} \ll |x_1^*|^\alpha$ and $|x_2^*| \gg |x_1^*|^\alpha,$ that is for points that neither Lemma \ref{h_1_more} nor Lemma \ref{h_para_bdry} apply to. (Lemma \ref{h_para_bdry} can be applied to points near the parabola but that do not get too close to its ``edge." Repeated applications of boundary Harnack could get similar results to hold in balls with centers of the form $\mathbf{u}^* = (u_1, u_2, 0,0)$ as long as $u_2$ is sufficiently small compared to $u_1.$) These bad points lie in a tube around the parabola of radius comparable to $|x_1^*|^\alpha.$  

Lemma \ref{h_para_bdry} shows that $h \not \approx 1$ for points close to the middle of the parabola, so along with Lemma \ref{his1}, this shows $\Z^4$ is not uniformly $S$-transient with respect to the parabola.

\section{Connections with Wiener's Test}\label{wiener_proof}

In many situations, Wiener's test gives an optimal way for determining classical transience/recurrence of a set $S \subset \Gamma$, where transience is taken to mean $\bP^x(X_n \in S \text{ i.o.}) =0$ and recurrence means $\bP^x(X_n \in S \text{ i.o.}) >0$. In many cases (such as for the SRW on $\Z^d$), a $0-1$ law holds for these probabilities, but such a $0-1$ law does not hold in the general setting considered in this paper. 

Below we give the version of Wiener's test in the case of interest to us. See, for example, \cite{Barlow_graphs, Lamperti_wiener, Revelle_Recurrence, Uchiyama_wiener} for statements of Wiener's test in various settings.

\begin{theorem}[Wiener's test for Harnack Graphs]\label{Wiener_Harnack}
Let $(\Gamma, \mathcal{K}, \pi)$ be a Harnack graph with controlled weights. Let $(X_n)_{n\geq0}$ denote the process on the graph.

Assume that $\Gamma$ is transient in the sense that $\bP^x (X_n = x \text{ i.o.}) =0$ for some/all $x \in \Gamma.$ Fix $o \in \Gamma$ and let $A_k := B_\Gamma (o, a^{k+1}) \setminus B_\Gamma(o, a^k)$ for some constant $a$.

Then there exists $a >1$ such that for any set $S \subset \Gamma,$ 
\begin{equation}\label{wiener2} \bP^o (X_n \in S \text{ i.o.}) = 0 \iff \sum_{k=1}^\infty \frac{\capac(S \cap A_k)}{\capac(A_k)} < + \infty.\end{equation}
Here $\capac$ denotes the capacity, defined as 
\[\capac(S) =  \sum_{y \in S} e_S(y),\]
where 
\[ e_S(y) =\begin{cases} \bP^y ( \forall n \geq 1, \ X_n \not \in S), & y \in S \\ 0, & y \not \in S\end{cases}\]
is the equilibrium potential of $S$.

Further, if $V$ denotes the volume function on $\Gamma,$ and $y \in A_k$ such that $d(y, \partial A_k) \approx a^k,$ then 
\begin{equation}\label{cap_ann} 
\capac(A_k) \approx \Big[\sum_{n=a^{2k}}^\infty \frac{1}{V(y, \sqrt{n})}\Big]^{-1}.
\end{equation}
\end{theorem}

This theorem follows by repeating the proof of Theorem 7.23 in \cite{Barlow_graphs} with a few modifications to account for the different form of assumed heat kernel bounds on $\Gamma$ here. 
 
There are key differences between Wiener's test and the questions we addressed in the main part of the paper. First, the definitions of transience used do not align. In this paper, we defined transience as $\psi_K(x) < 1$ for some/all $x \in \Gamma:= \widehat{\Gamma} \setminus K$. Wiener's test takes transience to be $\bP^x(X_n \in K \text{ i.o.}) =0$ for all $x \in \widehat{\Gamma}.$ These may not be the same, and Wiener's test does not account for uniform $S$-transience (see Example \ref{para_3} below), which is of much interest to us. Further, Wiener's test does not care about where the walk is started. However, we are only interested in starting the walk outside of the set $K.$ There may be cases where the random walk started well inside of $K$ is unlikely to ever leave $K,$ but a random walk started outside of $K$ may have a positive chance to never visit $K.$ 
 
 \begin{example}[Applying Wiener's test to the parabola example]\label{para_3} 
 Again, consider the ``parabola" $K$ inside a lattice $\Z^4$ as in Section \ref{para_Z4}.
 
 In the case of the lattice $\Z^d,$ we can take $a =2.$ We do this here to emphasize Theorem \ref{Wiener_Harnack} is a generalization of the classical formulation of Wiener's test for $\Z^d.$ 
 
 First, by (\ref{cap_ann}), we have 
 \[ \capac(A_k) \approx \Big[\sum_{n=2^{2k}}^\infty \frac{1}{V(y, \sqrt{n})}\Big]^{-1} \approx \Big[\sum_{n=2^{2k}}^\infty \frac{1}{n^2}\Big]^{-1} \approx 2^{2k}. \] 
 
The intersection of the parabola and the (4-dimensional) annulus, $K \cap A_K$, is contained inside a two-dimensional rectangle $R_k$ of length approximately $2^k$ and width approximately $2^{\alpha k},$ where $\alpha$ determines the shape of the parabola, i.e. $\alpha =1/2$ for an actual parabola. Since in $\Z^4$ the capacity of a point is a positive constant, if $|R_k|$ denotes the number of points in $R_k,$ then
 \begin{align*}
 \capac(S \cap A_k) \leq \capac(R_k) \leq c |R_k| \leq c 2^{k + k\alpha}.
 \end{align*} 
 
 Therefore
 \begin{align*}
 \sum_{k =0}^\infty \frac{\capac(S \cap Q_k)}{2^{2k}} \leq c \sum_{k =0}^\infty \frac{2^{k+k\alpha}}{2^{2k}} = c \sum_{k=0}^\infty \frac{1}{2^{(1-\alpha)k}} < \infty \quad \text{ since } \alpha \in (0,1).
 \end{align*}

Therefore $\Z^4 \setminus K$ is transient in the sense of Wiener's test and is $S$-transient, but it is not uniformly $S$-transient. This shows that Wiener's test is not sufficient for our purposes. 
 
 \end{example}

%%%%%%%%%%%%%%%%%%%%%%%%%%%%%%%%%%%%%%%%%%%%%%%%%%%%%%%%%%%%%%%%%%%
%%                                                               %%
%% Supplementary Material, if any, should be provided in         %%
%% {supplement} environment  with title and short description.   %%
%%                                                               %%
%%%%%%%%%%%%%%%%%%%%%%%%%%%%%%%%%%%%%%%%%%%%%%%%%%%%%%%%%%%%%%%%%%%

%\begin{supplement}
%\stitle{Title of Supplement A.}
%\sdescription{Short description of Supplement A.}
%\end{supplement}
%\begin{supplement}
%\stitle{Title of Supplement B.}
%\sdescription{Short description of Supplement B.}
%\end{supplement}

%%%%%%%%%%%%%%%%%%%%%%%%%%%%%%%%%%%%%%%%%%%%%%%%%%%%%%%%%%%%%%%%%%%
%%                                                               %%
%% Use the two commands below for producing your bibliography    %%
%% with bibtex, then comment again the commands and include the  %%
%% content of the .bbl file in this file below the commands.     %%
%%                                                               %%
%%%%%%%%%%%%%%%%%%%%%%%%%%%%%%%%%%%%%%%%%%%%%%%%%%%%%%%%%%%%%%%%%%%

%\bibliographystyle{amsplain}
%\bibliography{hittingprob_refs}

% add below the content of your .bbl file produced by bibtex.

\providecommand{\bysame}{\leavevmode\hbox to3em{\hrulefill}\thinspace}
\providecommand{\MR}{\relax\ifhmode\unskip\space\fi MR }
\providecommand{\MRhref}[2]{%
  \href{http://www.ams.org/mathscinet-getitem?mr=#1}{#2}
}
\providecommand{\href}[2]{#2}

\end{document}